\smallskip\documentclass[12pt]{article}
\usepackage[margin=1.25in]{geometry}
\usepackage{sectsty}
\usepackage{xcolor}
\usepackage{amssymb,amsmath}
\usepackage{amsthm}
\usepackage{hyperref}
\usepackage{mathrsfs}
\usepackage{textcomp}
\usepackage{enumitem}

\usepackage[abbrev,nobysame,non-compressed-cites]{amsrefs}
\usepackage{todonotes}

\hypersetup{
colorlinks,%
citecolor=black,%
filecolor=black,%
linkcolor=black,%
urlcolor=black
}

\usepackage{verbatim}
\numberwithin{equation}{section}

\makeatletter

\newdimen\bibspace
\makeatother

\makeatletter

\newtheorem{thm}{Theorem}[section]
\newtheorem{lem}[thm]{Lemma}

\newtheorem{defn}[thm]{Definition}

\newtheorem{cor}[thm]{Corollary}
\newtheorem{rem}[thm]{Remark}

\def\XXint#1#2#3{{\setbox0=\hbox{$#1{#2#3}{\int}$}
\vcenter{\hbox{$#2#3$}}\kern-.5\wd0}}

\def\P{\mathscr{P}}
\def\C{\mathscr{C}}

\newcommand{\al}{\alpha}                
\newcommand{\lda}{\lambda}
 
\newcommand{\om}{\Omega}                
\newcommand{\pa}{\partial}

\newcommand{\be}{\begin{equation}}      
\newcommand{\ee}{\end{equation}}

\newcommand{\ol}{\overline}

\newcommand{\R}{\mathbb{R}}

\newcommand{\wt}{\widetilde}

\allowdisplaybreaks

\begin{document}

\title{\textbf{Optimal regularity and fine asymptotics for very fast diffusion equations in bounded domains}
\bigskip}

\author{Tianling Jin\footnote{T. Jin was partially supported by NSFC grant 12122120, and Hong Kong RGC grants GRF 16304125, GRF 16303624 and GRF 16303822.}, \quad Xushan Tu, \quad
Jingang Xiong\footnote{J. Xiong was partially supported by NSFC grants 12325104.}, \quad Zhen Zheng}

\date{\today}

\maketitle

\begin{abstract}
We prove the optimal global regularity of admissible solutions to a
transformed very fast diffusion equation in the range $-1<p<0$, posed
on smooth bounded domains with zero Dirichlet boundary data and initial
data comparable to the distance function. More precisely, we establish
existence and uniqueness and show that solutions belong to
$C^{1,p+1}(\ol\om)$ in space for every positive time and are
$C^\infty$ in time uniformly up to the boundary. Moreover, all their
time derivatives belong to $C^{1,p+1}(\ol\om)$, and the exponent
$p+1$ is optimal. These regularity estimates further yield fine
long-time asymptotics toward the friendly giant solution, including a
first-order expansion in the $C^{1,p+1}(\ol\om)$ topology and an
improved convergence rate for the relative error in
$C^{p+1}(\ol\om)$.


\medskip


\end{abstract}

\section{Introduction}
Let $\Omega\subset\mathbb{R}^{n}$, $n\geq1$, be a bounded domain with
smooth boundary. In this paper, we study the following
Cauchy--Dirichlet problem:
\begin{equation}\label{eq:equation for v}
\begin{cases}
\partial_t(v^p)-p\Delta v=0
&\quad \text{in }\Omega\times(0,T),\\
v=0
&\quad \text{on }\partial\Omega\times(0,T),\\
v(\cdot,0)=v_0
&\quad \text{on }\overline{\Omega},
\end{cases}
\end{equation}
where $-1<p<0$, $T>0$ and $v_0$ is nonnegative and nontrivial. Since $p<0$, the quantity
$v^p$ becomes singular as $v$ approaches zero. It is therefore
important to specify both the boundary behavior of $v$ and the
integrability of $v^p$. We work with the following class of solutions.

\begin{defn}\label{def:admissible-classical-solution}
Let $v_0\in C(\overline{\Omega})$ be positive in $\Omega$ and satisfy
\[
v_0=0\quad\text{on }\partial\Omega,
\qquad
v_0^p\in L^1(\Omega).
\]
A nonnegative function $v$ is called an admissible solution of
\eqref{eq:equation for v} with initial data $v_0$ if
\begin{enumerate}
\item[(i)] $v\in C_{x,t}^{2,1}(\Omega\times(0,T))
\cap C(\overline{\Omega}\times[0,T))$ and
$v>0$ in $\Omega\times(0,T)$;
\item[(ii)] $v^p\in C([0,T);L^1(\Omega))$;
\item[(iii)] $v$ satisfies \eqref{eq:equation for v} in the classical sense.
\end{enumerate}
\end{defn}

Notice that the condition
$v^p\in C([0,T);L^1(\Omega))$ is not automatic: because $p<0$, the
function $v^p$ blows up at the lateral boundary.

Throughout the paper, we write
$
d(x):=\operatorname{dist}(x,\pa\om).
$
We first state our main results. We prescribe the
natural linear vanishing (or non-degenerate condition) of the initial data at the boundary: assume
$v_0\in C(\ol\om)$ is positive in $\om$ and satisfies
\be\label{eq:inequality for v0}
\frac{1}{c_0}
\leq \inf_\om\frac{v_0}{d}
\leq \sup_\om\frac{v_0}{d}
\leq c_0
\ee
for some constant $c_0\geq1$. Notice
that, since $p>-1$, it implies $v_0^p\in L^1(\om)$. We establish the global well-posedness and the optimal boundary regularity of admissible solutions with such initial data. Furthermore, the solution will preserve the non-degenerate condition.

\begin{thm}\label{thm:main thm 1}
Let $\om$ be a bounded smooth domain, let $p\in(-1,0)$, and let
$v_0\in C(\ol\om)$ be positive in $\om$ and satisfy
\eqref{eq:inequality for v0} for some constant $c_0\geq1$. Then
\eqref{eq:equation for v} admits a unique global admissible solution
$v$ with $T=+\infty$. Moreover,
\[
v(x,\cdot)\in C^\infty\bigl((0,\infty)\bigr)\qquad \text{uniformly~for~} x\in \ol\om
\]
and
\be\label{eq:main thm 1 regularity in x}
\pa_t^k v(\cdot,t)\in C^{1,p+1}(\ol\om)
\quad\text{for every }t>0
\text{ and }k\in\mathbb N\cup\{0\}.
\ee
The exponent $p+1$ in
\eqref{eq:main thm 1 regularity in x} is optimal. Finally, for every $s>0$, there exist constants
$0<c_s\leq C_s<+\infty$, depending only on
$p,s,c_0,C_0,n$, and $\|\Omega\|_{C^2}$, such that
\be\label{eq:comparable to distance function0}
c_sd(x)\leq v(x,t)\leq C_sd(x),
\qquad (x,t)\in\Omega\times[0,s].
\ee
\end{thm}

Without this non-degenerate condition \eqref{eq:inequality for v0}, the initial pressure may vanish too rapidly near the boundary or possess uncontrolled zero sets, and such degeneracies may not be removed by the evolution. Indeed, for related pressure equations, zero sets may carry additional trace information, and the same continuous initial datum may admit distinct continuations or infinitely many solutions with different shrinking interfaces
\cites{chasseigne2003pressure,bertsch1992nonuniqueness,vazquez2003darcy}. Thus, Theorem \ref{thm:main thm 1} may fail without \eqref{eq:inequality for v0}.

The regularity exponent in Theorem \ref{thm:main thm 1} is dictated by a special separable solution. Let $S$ be the unique nontrivial
nonnegative solution of the singular elliptic problem
\be\label{eq:equation of S}
-\Delta S=\frac{1}{1-p}S^p
\quad\text{in }\om,
\qquad
S=0
\quad\text{on }\pa\om.
\ee
Then
\be\label{eq:giant solution}
U(x,t):=t^{\frac{1}{p-1}}S(x),
\qquad (x,t)\in\om\times(0,\infty),
\ee
is an exact solution of \eqref{eq:equation for v}. It is usually
referred to as the \emph{friendly giant solution}.
Gui--Lin~\cite{gui1993regularity} proved that
\[
S\in C^{1,p+1}(\ol\om).
\]
More precisely, the boundary expansion underlying their result shows
that, in general,
\[
S\notin C^{1,p+1+\varepsilon}(\ol\om)
\qquad\text{for any }0<\varepsilon<-p.
\]
Since $U(\cdot,t)$ has exactly the same spatial regularity as $S$ for
every $t>0$, the exponent $p+1$ in
\eqref{eq:main thm 1 regularity in x} cannot in general be improved.

Our second main result gives a fine description of the long-time
behavior. In addition to the leading-order convergence to $S$, it
identifies the first correction term and yields convergence to a
suitable time translation of the friendly giant solution.

\begin{thm}\label{main thm 1.3}
Let $\om\subset\R^n$ be a bounded smooth domain, let $p\in(-1,0)$,
and let $v_0\in C(\ol\om)$ be positive in $\om$ and satisfy
\eqref{eq:inequality for v0}. Let $v$ be the admissible solution of
\eqref{eq:equation for v} with $T=+\infty$. Then there exist constants
$ A_1\geq0, \tau^* \geq 0, C_1>1 $ and $\gamma_1>0$
such that, for every $\delta>0$,
\be\label{eq:expansion of v in long time}
\left\|
t^{\frac{1}{1-p}}v(\cdot,t)
-S+\frac{A_1}{t}S
\right\|_{C^{1,p+1}(\ol\om)}
\leq \frac{C_1}{t^{1+\gamma_1}}
\quad\text{for all }t>\delta,
\ee
and
\be\label{eq:long time convergence of u}
\left\|
\frac{v(\cdot,t)}{U(\cdot,t+\tau^*)}-1
\right\|_{C^{p+1}(\ol\om)}
\leq \frac{C_1}{t^{1+\gamma_1}}
\quad\text{for all }t>\delta.
\ee
Here $S$ and $U$ are given by \eqref{eq:equation of S} and
\eqref{eq:giant solution}, respectively.

The constant $\gamma_1$ depends only on $n$, $p$, and $\om$, while
$C_1$ depends only on $n$, $p$, $\delta$, $\om$, and $c_0$.
The coefficient $A_1$, and hence the time shift $\tau^*$, may depend
on $v_0$, but their sizes are controlled in terms of
$n$, $p$, $\om$, and $c_0$.

Furthermore, for every $l\in\mathbb N$, there exists a constant
$C_2>0$, depending only on $n$, $p$, $\delta$, $\om$, $l$, and $c_0$,
such that
\be\label{eq:decay rate of v}
\left\|
\pa_t^l v(\cdot,t)
\right\|_{C^{1,p+1}(\ol\om)}
\leq C_2t^{-\frac{1}{1-p}-l}
\quad\text{for all }t>\delta.
\ee
\end{thm}

The quotient in \eqref{eq:long time convergence of u} is understood
through its continuous extension to $\ol\om$. Such an extension is
well defined because both $v(\cdot,t)$ and $U(\cdot,t+\tau^*)$ vanish
linearly on $\pa\om$.

To place Theorems \ref{thm:main thm 1} and
\ref{main thm 1.3} in a broader context, we first observe that,
for positive solutions, equation \eqref{eq:equation for v}
can equivalently be written as
\begin{equation}\label{eq:nondivergence-form-intro}
v_t=v^{1-p}\Delta v.
\end{equation}
Hence the diffusion coefficient vanishes at the lateral boundary. If
we set $\alpha:=1-p,$
then $\alpha\in(1,2)$, and \eqref{eq:nondivergence-form-intro} becomes
\[
v_t=v^\alpha\Delta v.
\]
Consequently, although the equation is uniformly parabolic on every
compact subset of $\Omega\times(0,+\infty)$, it is strongly degenerate
near $\partial\Omega$. If a Hopf-type boundary behavior
$
v(x,t)\asymp d(x),
$
is available, then the leading coefficient behaves like
$d(x)^\alpha$. Moreover, $v^p\asymp d(x)^p$, and
\[
d^p\in L^1(\Omega)
\quad\Longleftrightarrow\quad
p>-1.
\]
Thus, the range $-1<p<0$ is naturally compatible with the expected
linear vanishing of admissible solutions at the boundary.

The equation belongs to the general family of nonlinear filtration
equations:
\begin{equation}\label{eq:filtration-intro}
u_t=\Delta\Phi_m(u),
\qquad
\Phi_m(s):=
\begin{cases}
\dfrac{s^m}{m},&m\neq0,\\[2mm]
\log s,&m=0.
\end{cases}
\end{equation}
The normalization in \eqref{eq:filtration-intro} has the advantage that
$\Phi_m'(s)=s^{m-1}>0$ for every $m\in\mathbb{R}$. The cases $m>1$,
$m=1$, $0<m<1$, $m=0$, and $m<0$ are commonly referred to,
respectively, as the porous medium, linear heat, fast diffusion,
logarithmic diffusion, and very fast or super-fast diffusion regimes.
Equivalently, in the formulation for $v=u^m$:
\[
\partial_t(v^p)-p\Delta v=0,\qquad p=\frac1m,
\]
the porous medium and fast diffusion ranges correspond to
$0<p<1$ and $p>1$, respectively. The present problem corresponds to
$m=1/p<-1$, and therefore lies in the negative-power, super-fast
diffusion range. Notice also that the boundary condition $v=0$ is
formally transformed into the infinite boundary condition $u=+\infty$.
This is one reason why it is preferable to formulate the problem in
terms of $v$.

Nonlinear filtration equations arise from density-dependent diffusion
and Darcy-type laws and have applications in flows through porous
media, gas kinetics, plasma physics, thin-film dynamics, and geometric
evolution equations. We refer to
\cites{daskalopoulos2007degenerate,vazquez2006smoothing,vazquez2007porous}
for systematic accounts of their analytical and physical background.
The negative-power regime is substantially more delicate: existence
may depend sensitively on the class of data and on the behavior at
spatial infinity or at the boundary. Some foundational results
on very fast diffusion, nonexistence, and extended solutions can be
found in
\cites{daskalopoulos1997nonlinear,vazquez1992nonexistence,
chasseigne2002theory,bonforte2010positivity}.

We next recall some characteristic properties of the fast diffusion
range. For $0<m<1$, the diffusivity $u^{m-1}$ becomes singular near the
zero set. As a result, nontrivial nonnegative solutions typically
exhibit infinite speed of propagation: compactly supported initial data
produce solutions that become positive everywhere in the interior for
every positive time. This is accompanied by strong $L^q$--$L^\infty$
smoothing effects and intrinsic Harnack inequalities; see, among many
others,
\cites{benilan1981regularizing,herrero1985cauchy,
bonforte2006global,bonforte2010positivity,vazquez2006smoothing}.
For the homogeneous Dirichlet problem in a bounded domain, solutions
vanish in finite time. If $T^*$ denotes the extinction time and
$v=u^m$, then, in the usual subcritical Dirichlet range, the global
Harnack principle gives the two-sided estimate
\be\label{eq:global harnack inequality}
0<
\inf_{\Omega}\frac{v}{d}
\leq
\sup_{\Omega}\frac{v}{d}
<\infty,
\ee
away from the initial time. Thus, the flux variable $v$ vanishes
linearly at the boundary. Quantitative positivity estimates, global
Harnack principles, and their relation to extinction behavior were
developed in
\cites{dibenedetto1991local,bonforte2006global,
bonforte2012behaviour,bonforte2024cauchy}.

The estimate \eqref{eq:global harnack inequality} determines the size of the solution
near the boundary, but by itself does not yield differentiability of
$v/d$. Obtaining sharp boundary regularity is more subtle because the
equation loses uniform parabolicity precisely where the Dirichlet
condition is imposed. Jin--Xiong proved optimal boundary regularity
in the subcritical and critical Sobolev regimes
\cite{jin2023optimal} and subsequently established global H\"older
gradient estimates and optimal global regularity for bounded positive
 solutions \cite{jin2025regularity}. In particular, when $p>1$ is
not an integer, the natural spatial regularity of the flux variable is
$C^{2+p}(\overline{\Omega})$, and this exponent is generally sharp;
for integral $p$, smoothness up to the boundary can be recovered.
These regularity results also play an important role in the analysis of
extinction profiles and convergence rates. This topic has been extensively studied in, e.g., 
\cites{berryman1980stability,feireisl2000convergence,
bonforte2021sharp,akagi2016stability,akagi2023rates,
choi2023asymptotics,choi2024finitedimensional,jin2026bubbledynamics}.

The porous medium range presents the opposite propagation mechanism.
For $m>1$, the diffusivity vanishes at zero, and compactly supported
solutions propagate with finite speed. Their positivity sets therefore
have moving free boundaries. The theory of weak solutions, local
H\"older regularity, and regularizing effects goes back to
\cites{oleinik1958cauchy,aronson1979regularite,
dibenedetto1993degenerate,vazquez2007porous}. The regularity of
solutions and free boundaries was developed through a series of works,
including
\cites{caffarelli1980regularity,caffarelli1987lipschitz,
caffarelli1990regularity,daskalopoulos1998regularity,
daskalopoulos2001all,kienzler2018flatness}. Higher regularity of the
interface generally requires nondegeneracy, flatness, or suitable
regularity and compatibility assumptions on the initial data.

For the porous medium equation in a bounded domain with zero Dirichlet
data, the behavior becomes more regular after the support has reached
the boundary. Large-time estimates and convergence to the separable
``friendly giant'' profile were studied in
\cites{aronson1981large,vazquez2004dirichlet}. More recently, Jin--Ros-Oton--Xiong  \cite{jin2024optimal} proved the optimal eventual regularity
\[
u^m(\cdot,t)\in C^{2+\frac1m}(\overline{\Omega})
\]
after a suitable waiting time, together with smoothness in time and
refined large-time asymptotics. In terms of the
flux variable $v=u^m$ and $p=1/m\in(0,1)$, this is precisely the
$C^{2+p}$ scale. Thus, the optimal boundary regularity $C^{2+p}$
appears naturally in both the fast diffusion range $p>1$ and the
porous medium range $0<p<1$, although the mechanisms leading to it are
quite different.

The linearization of \eqref{eq:nondivergence-form-intro} reflects the
same boundary geometry. If $\bar v$ is a reference solution and
$v=\bar v+\varepsilon\phi$, then the first-order linearized equation is
\begin{equation}\label{eq:linearization-intro}
\phi_t-\bar v^\alpha\Delta\phi
-\alpha\bar v^{\alpha-1}(\Delta\bar v)\phi=0,
\qquad
\alpha=1-p\in(1,2).
\end{equation}
When $\bar v\asymp d$, the principal part of
\eqref{eq:linearization-intro} is locally modeled, after flattening the
boundary, by
\[
\partial_t-x_n^\alpha a_{ij}(x,t)D_{ij}.
\]
Weighted Sobolev solvability and regularity for linear parabolic and
elliptic equations with singular or degenerate coefficients have been
developed by Dong and his collaborators in both divergence and
nondivergence forms
\cites{dong2021parabolic,dong2023parabolic,
dong2023degenerate,dong2024nondivergence}. In particular, the works of
Dong, Phan, and Tran treat leading coefficients comparable to
$x_n^\alpha$ for $\alpha\in(0,2)$, a range that includes
$\al=1-p$. More recent developments include boundary
Schauder-type estimates under partially Dini mean oscillation
assumptions and weighted Sobolev estimates for strongly degenerate
operators
\cites{dong2025schaudertypeestimatesdegenerate,
dong2026degenerate,dong2025nondivergence,dong2026sobolev}.
A Schauder theory tailored specifically to linearized very fast
diffusion equations in bounded domains has also been developed in
\cite{JTXZ}. These linear results provide important guidance, although
the coefficient $\bar v^\alpha$ in
\eqref{eq:linearization-intro} is generated by the unknown nonlinear
solution and must be controlled simultaneously with its boundary
behavior.
For related boundary regularity and solvability results in the fully
nonlinear porous medium setting, we refer to Lee--Yun
\cite{lee2025boundary} and Yun \cite{yun2024regularity}.

The remainder of the paper is organized as follows. In Section~\ref{sec:existence}, we adapt the general
$L^1$-contraction argument to our equation, derive the comparison
principle, and prove the existence and uniqueness of admissible
solutions. In Section~\ref{sec:C1alpha}, we develop
boundary $C^{1,\gamma}$ estimates for the linear degenerate model
equations and then apply them to the nonlinear equation to obtain
global boundary regularity. In Section~\ref{sec:regularity}, we
introduce the intrinsic weighted H\"older and Schauder spaces and
derive the higher-order estimates leading to smoothness in time and
optimal $C^{1,p+1}$ spatial regularity. Finally, we study the long-time
behavior of the solution and establish the refined asymptotic
expansion and decay estimates.

\medskip

\textbf{AI statement:} All the mathematical results, overall proof architecture, and key arguments were developed by the authors. During later-stage revisions, AI tools were used to assist with editing and with simplifying some auxiliary arguments. All AI-assisted modifications were independently reviewed and verified by the authors, who take full responsibility for the final manuscript. 

\section{Existence and uniqueness of the solution}\label{sec:existence}

Let $n\ge2$, $-1<p<0$, and $R>0$. We write
$x=(x',x_n)\in\R^{n-1}\times\R$, and let $x_0\in\R^n$ and $t_0\in\R$.
The open ball and open half-ball are defined by
\[
B_R(x_0):=\bigl\{x\in\R^n \mid |x-x_0|<R\bigr\},
\qquad
B_R^+(x_0):= B_R(x_0)\cap\bigl\{x\in\R^n \mid x_n>0\bigr\}.
\]
The corresponding backward parabolic cylinders are defined by
\[
Q_R(x_0,t_0):= B_R(x_0)\times\bigl(t_0-R^{p+1},\,t_0\bigr],
\qquad
Q_R^+(x_0,t_0):= B_R^+(x_0)\times\bigl(t_0-R^{p+1},\,t_0\bigr].
\]
We denote the parabolic boundary of $Q_R(x_0,t_0)$ by
$\partial_p Q_R(x_0,t_0)$; namely,
\[
\partial_p Q_R(x_0,t_0)
:= \bigl(\partial B_R(x_0)\times (t_0-R^{p+1},\,t_0]\bigr)
\,\bigcup\, \bigl(B_R(x_0)\times\{t=t_0-R^{p+1}\}\bigr).
\]
When $(x_0,t_0)=(0,0)$, we use the abbreviations
\[
B_R:=B_R(0),\quad B_R^+:=B_R^+(0),\quad
Q_R:=Q_R(0,0),\quad Q_R^+:=Q_R^+(0,0).
\]

For $\alpha\in(0,1)$, we define the H\"older seminorm with respect to
the spatial variable by
\[
[u]_{C_x^\alpha(\Omega\times(0,T])}
:=
\sup_{t\in(0,T]}\sup_{\substack{x,y\in\Omega\\ x\neq y}}
\frac{|u(x,t)-u(y,t)|}{|x-y|^\alpha}.
\]
We say that $u\in C_x^\alpha(\Omega\times(0,T])$ if
\[
\|u\|_{L^\infty(\Omega\times(0,T])}
+
[u]_{C_x^\alpha(\Omega\times(0,T])}
<\infty.
\]

Similarly, the H\"older seminorm with respect to the time variable is
defined by
\[
[u]_{C_t^\alpha(\Omega\times(0,T])}
:=
\sup_{x\in\Omega}\sup_{\substack{s,t\in(0,T]\\ s\neq t}}
\frac{|u(x,t)-u(x,s)|}{|t-s|^\alpha}.
\]
We say that $u\in C_t^\alpha(\Omega\times(0,T])$ if
\[
\|u\|_{L^\infty(\Omega\times(0,T])}
+
[u]_{C_t^\alpha(\Omega\times(0,T])}
<\infty.
\]

Finally, we define
\[
C_{x,t}^{\alpha,\beta}(\Omega\times(0,T])
:=
C_x^\alpha(\Omega\times(0,T])
\cap
C_t^{\beta}(\Omega\times(0,T]),
\]
equipped with the norm
\[
\|u\|_{C_{x,t}^{\alpha,\beta}(\Omega\times(0,T])}
:=
\|u\|_{L^\infty(\Omega\times(0,T])}
+
[u]_{C_x^\alpha(\Omega\times(0,T])}
+
[u]_{C_t^{\beta}(\Omega\times(0,T])}.
\]
For notational convenience, we use \(\|\Omega\|_{C^2}\) or
$\|\Omega\|_{C^{2,\alpha}}$ to denote collectively the \(C^2\) or
$C^{2,\alpha}$-character, respectively, together with the diameter of
\(\Omega\).

\begin{thm}[$L^1$-contraction principle]\label{thm:L1-contraction}
Let $v$ and $w$ be two admissible solutions of
\eqref{eq:equation for v}, with initial data $v_0$ and $w_0$,
respectively. Then, for every $t\in[0,T)$,
\begin{equation}\label{eq:L1-contraction}
\int_\Omega
\bigl[v^p(x,t)-w^p(x,t)\bigr]_+\,dx
\leq
\int_\Omega
\bigl[v_0^p(x)-w_0^p(x)\bigr]_+\,dx .
\end{equation}
\end{thm}

\begin{proof}
The proof is based on Otto's time-doubling argument for
$L^1$-contraction \cite{otto1996contraction}, in the form adapted by
Hissink Muller and Sonner
\cite{hissinkmuller2022wellposedness}*{Lemma 4.1}.
For completeness, we provide the details in the present setting.

We use two regularization parameters: $\delta$ for the approximation
of the positive-sign function and $\varepsilon$ for convolution in
the two time variables. For every $\delta>0$, choose a nondecreasing
function $\eta_\delta\in C^\infty(\mathbb R)$ such that
\[
0\leq\eta_\delta\leq1,
\qquad
\eta_\delta(r)=0\quad\text{for }r\leq\delta,
\qquad
\eta_\delta(r)=1\quad\text{for }r\geq2\delta.
\]
Then $\eta_\delta'\geq0$, $\eta_\delta(0)=0$, and
\[
\eta_\delta(r)\longrightarrow\operatorname{sgn}_{+}(r)
\qquad\text{for every }r\in\mathbb R.
\]

For $a,b>0$, define
\begin{align*}
\Phi_\delta(a,b)
&:=\int_b^a
p r^{p-1}\eta_\delta\bigl(p(r-b)\bigr)\,dr,\\
\Psi_\delta(a,b)
&:=-\int_a^b
p r^{p-1}\eta_\delta\bigl(p(a-r)\bigr)\,dr.
\end{align*}
Since $p<0$, both functions vanish when $a\geq b$. When $a<b$, the
bounds $0\leq\eta_\delta\leq1$ imply
\begin{equation}\label{eq:PhiPsi-bound}
0\leq \Phi_\delta(a,b),\,\Psi_\delta(a,b)
\leq a^p-b^p=[a^p-b^p]_+.
\end{equation}
Moreover,
\begin{equation}\label{eq:PhiPsi-limit}
\Phi_\delta(a,b),\,\Psi_\delta(a,b)
\longrightarrow[a^p-b^p]_+
\qquad\text{as }\delta\downarrow0.
\end{equation}

For every $s,t\in(0,T)$, we have
\begin{align*}
\frac{d}{dt}\int_\Omega
\Phi_\delta\bigl(v(x,t),w(x,s)\bigr)\,dx
&=
\int_{\om}
\partial_t(v^p)
\eta_\delta\bigl(p(v(t)-w(s))\bigr)\,dx,\\
\frac{d}{ds}\int_\Omega
\Psi_\delta\bigl(v(x,t),w(x,s)\bigr)\,dx
&=
-\int_{\om}
\partial_s(w^p)
\eta_\delta\bigl(p(v(t)-w(s))\bigr)\,dx.
\end{align*}
Combining these identities with the equations satisfied by $v(t)$ and
$w(s)$ yields
\begin{align}
&\frac{d}{dt}\int_\Omega \Phi_\delta(v(t),w(s))\,dx
+\frac{d}{ds}\int_\Omega \Psi_\delta(v(t),w(s))\,dx \notag\\
&\quad
+p^2\int_\Omega
\eta_\delta'\bigl(p(v(t)-w(s))\bigr)
\lvert\nabla v(t)-\nabla w(s)\rvert^2\,dx
=0.
\label{eq:doubled-regularized}
\end{align}

Let $\varphi\in C_c^\infty((0,T)^2)$ be nonnegative. Multiplying
\eqref{eq:doubled-regularized} by $\varphi(t,s)$, integrating with
respect to $(t,s)$, and integrating the first two terms by parts, we
obtain, after dropping the nonnegative spatial term,
\begin{align}
-\int_0^T\!\int_0^T\!\int_\Omega
\bigl(
\Phi_\delta(v(t),w(s))\,\partial_t\varphi(t,s)
+\Psi_\delta(v(t),w(s))\,\partial_s\varphi(t,s)
\bigr)
\,dx\,dt\,ds
\leq0.
\label{eq:doubled-before-limit}
\end{align}
By \eqref{eq:PhiPsi-bound},
\[
0\leq \Phi_\delta(v(t),w(s)),\,\Psi_\delta(v(t),w(s))
\leq v^p(t)+w^p(s),
\]
and the right-hand side is integrable on
$\Omega\times(0,T)^2$. Therefore, by
\eqref{eq:PhiPsi-limit} and the dominated convergence theorem, we may
let $\delta\to 0$ in \eqref{eq:doubled-before-limit} to obtain
\begin{align}
-\int_0^T\!\int_0^T\!\int_\Omega
[v^p(x,t)-w^p(x,s)]_+
\bigl(\partial_t\varphi(t,s)+\partial_s\varphi(t,s)\bigr)
\,dx\,dt\,ds
\leq0.
\label{eq:doubled-kato}
\end{align}

Let $\rho\in C_c^\infty((-1,1))$ be nonnegative and even, with
$\int_\mathbb R\rho=1$, and define
\[
\rho_\varepsilon(r)
:=\frac1\varepsilon\rho\!\left(\frac r\varepsilon\right).
\]
For an arbitrary nonnegative function
$\zeta\in C_c^\infty((0,T))$, choose
\[
\varphi_\varepsilon(t,s)
:=\rho_\varepsilon(t-s)
\zeta\!\left(\frac{t+s}{2}\right).
\]
For sufficiently small $\varepsilon$,
\[
(\partial_t+\partial_s)\varphi_\varepsilon(t,s)
=
\rho_\varepsilon(t-s)
\zeta'\!\left(\frac{t+s}{2}\right).
\]
Since
$v^p,w^p\in C([0,T);L^1(\Omega))$, we may let
$\varepsilon\to 0$ in \eqref{eq:doubled-kato} with this choice
of $\varphi_\varepsilon$. It follows that
\begin{equation}\label{eq:distributional-monotonicity}
-\int_0^T
\left(
\int_\Omega[v^p(x,t)-w^p(x,t)]_+\,dx
\right)\zeta'(t)\,dt
\leq0
\end{equation}
for every nonnegative $\zeta\in C_c^\infty((0,T))$.

Consequently, the continuous function
\[
F(t):=\int_\Omega[v^p(x,t)-w^p(x,t)]_+\,dx
\]
satisfies $F'\leq0$ in the distributional sense and is therefore
nonincreasing on $[0,T)$. Hence $F(t)\leq F(0)$ for every
$t\in[0,T)$, which is precisely \eqref{eq:L1-contraction}.
\end{proof}

As an immediate consequence, we obtain the following comparison
principle.

\begin{thm}[Comparison principle]\label{thm:comparison}
Let $v$ and $w$ be as in Theorem~\ref{thm:L1-contraction}. If
$v_0\leq w_0$ a.e.\ in $\Omega$, then
\[
v(\cdot,t)\leq w(\cdot,t)
\qquad\text{a.e.\ in }\Omega
\]
for every $t\in[0,T)$. In particular, an admissible solution with
prescribed initial data is unique.
\end{thm}

Finally, we establish the existence of admissible solutions for a
class of initial data.

\begin{thm}\label{thm:existence}
Let $T>0$, and assume that $v_0\in C(\overline{\Omega})$ satisfies
\[
c_0d(x)\leq v_0(x)\leq C_0d(x)
\qquad \text{in } \Omega
\]
for some constants $0<c_0\leq C_0<+\infty$. Then
\eqref{eq:equation for v} admits a unique admissible solution $v$ on
$\Omega\times[0,T]$. Moreover, there exist constants
$0<c_T\leq C_T<+\infty$, depending only on
$p,T,c_0,C_0,n$, and $\|\Omega\|_{C^2}$, such that
\be\label{eq:comparable to distance function}
c_Td(x)\leq v(x,t)\leq C_Td(x),
\qquad (x,t)\in\Omega\times[0,T].
\ee
\end{thm}

\begin{proof}
We first rewrite \eqref{eq:equation for v} as
\begin{equation}\label{eq:expand equation for v}
\begin{cases}
v_t-v^{1-p}\Delta v=0, & \text{in } \Omega\times(0,T],\\
v=0, & \text{on } \partial\Omega\times(0,T],\\
v(\cdot,0)=v_0, & \text{on } \ol\Omega,
\end{cases}
\end{equation}
by decreasing $T$ a little bit if needed. The assumptions on $v_0$ and its continuity imply that
$v_0=0$ on $\partial\Omega$. For each $\varepsilon\in(0,1)$, consider
the regularized problem
\begin{equation}\label{eq:positive boundary for v}
\begin{cases}
v_t-v^{1-p}\Delta v=0, & \text{in } \Omega\times(0,T],\\
v=\varepsilon, & \text{on } \partial\Omega\times(0,T],\\
v(\cdot,0)=v_0+\varepsilon, & \text{on } \ol\Omega.
\end{cases}
\end{equation}
By the standard theory for uniformly parabolic quasilinear equations
\cite{ladyzenskaja1967linear}, using an approximation of the initial
data if necessary, \eqref{eq:positive boundary for v} admits a unique
classical solution
\[
v_\varepsilon\in C_{x,t}^{2,1}(\Omega\times(0,T])
\cap C(\ol\om\times[0,T]).
\]
Moreover, the maximum principle gives
\[
\varepsilon\leq v_\varepsilon(x,t)
\leq\|v_0\|_{L^\infty(\Omega)}+\varepsilon
\qquad\text{in }\ol\Omega\times[0,T].
\]

We next derive estimates independent of $\varepsilon$. Let $\lambda_1$
be the principal Dirichlet eigenvalue of $-\Delta$ in $\om$, and let
$\phi$ be a corresponding positive eigenfunction normalized by
$\|\phi\|_{L^\infty(\om)}=1$:
\[
-\Delta\phi=\lambda_1\phi \quad\text{in }\om,
\qquad
\phi=0 \quad\text{on }\pa\om,
\qquad
\phi>0 \quad\text{in }\om.
\]
By Hopf's lemma, there exist constants $c_\phi,C_\phi>0$ such that
\[
c_\phi d(x)\leq\phi(x)\leq C_\phi d(x)
\qquad\text{in }\om.
\]

Set $C=C_0/c_\phi$ and
\[
w(x)=\varepsilon+C\phi(x).
\]
Then
\[
w_t-w^{1-p}\Delta w
=
\lambda_1(\varepsilon+C\phi)^{1-p}C\phi
\geq0.
\]
Moreover, $w=\varepsilon=v_\varepsilon$ on the lateral boundary and
\[
w(x)\geq v_0(x)+\varepsilon
\qquad\text{at }t=0.
\]
The comparison principle therefore yields
\begin{equation}\label{eq:upper barrier}
v_\varepsilon(x,t)\leq\varepsilon+C\phi(x)
\quad\text{in }\Omega\times(0,T].
\end{equation}

For the lower bound, set
\[
c=
\min\left\{
\frac{c_0}{C_\phi},
\bigl((1-p)\lambda_1\bigr)^{-\frac{1}{1-p}}
\right\}
\]
and define
\[
z(x,t)=c\phi(x)(1+t)^{\frac{1}{p-1}}.
\]
A direct computation gives
\[
\begin{aligned}
z_t-z^{1-p}\Delta z
&=(1+t)^{\frac{2-p}{p-1}}
\left[
-\frac{c\phi}{1-p}
+\lambda_1(c\phi)^{2-p}
\right]
\leq0.
\end{aligned}
\]
Thus, $z$ is a subsolution of
\eqref{eq:positive boundary for v}. Since
\[
z=0\leq\varepsilon=v_\varepsilon
\qquad\text{on }\pa\om\times(0,T]
\]
and
\[
z(x,0)=c\phi(x)
\leq v_0(x)
\leq v_0(x)+\varepsilon,
\]
the comparison principle gives
\be\label{eq:low barrier}
v_\varepsilon(x,t)\geq
c\phi(x)(1+t)^{\frac{1}{p-1}}
\quad\text{in }\om\times(0,T].
\ee

We now pass to the limit as $\varepsilon\to0$. Let
\[
K\Subset K'\Subset\Omega\times(0,T].
\]
By \eqref{eq:upper barrier} and \eqref{eq:low barrier}, there exist
constants $0<m_K\leq M_K<\infty$, independent of $\varepsilon$, such
that
\[
m_K\leq v_\varepsilon\leq M_K
\qquad\text{in }K'.
\]
Hence the equations are uniformly parabolic on $K'$. The interior
H\"older estimates, followed by the parabolic Schauder estimates, give
uniform $C_{x,t}^{2+\gamma,1+\gamma/2}$ estimates for
$v_\varepsilon$ on $K$. Therefore, by the Arzel\`a--Ascoli theorem and
a diagonal argument, there exist a sequence $\varepsilon_j\to0$ and a
function $v$ such that
\[
v_{\varepsilon_j}\to v,\qquad
(v_{\varepsilon_j})_t\to v_t,\qquad
D^2v_{\varepsilon_j}\to D^2v
\]
locally uniformly in $\Omega\times(0,T]$. Passing to the limit gives
\[
v_t-v^{1-p}\Delta v=0
\qquad\text{in }\Omega\times(0,T].
\]

It remains to identify the initial and boundary values. Let
$\Omega'\Subset\Omega''\Subset\Omega$. The barrier estimates imply
that the equations for $v_\varepsilon$ are uniformly parabolic in a
neighborhood of $\Omega'\times\{0\}$, with ellipticity constants
independent of $\varepsilon$. Since the family
$\{v_0+\varepsilon\}_{0<\varepsilon<1}$ has a common modulus of
continuity, the standard local initial-time continuity estimate gives
\[
\lim_{t\to0}
\sup_{0<\varepsilon<1}
\left\|
v_\varepsilon(\cdot,t)-(v_0+\varepsilon)
\right\|_{L^\infty(\Omega')}
=0.
\]
Passing to the limit along $\varepsilon_j\to0$, we obtain
\[
v(\cdot,t)\to v_0
\quad\text{locally uniformly in }\Omega
\quad\text{as }t\to0.
\]
We therefore define $v(\cdot,0)=v_0$.

Letting $\varepsilon_j\to0$ in \eqref{eq:upper barrier} and
\eqref{eq:low barrier}, we obtain
\[
c\phi(x)(1+t)^{\frac{1}{p-1}}
\leq v(x,t)\leq C\phi(x)
\qquad\text{in }\Omega\times(0,T].
\]
Together with the corresponding bounds for $v_0$, this yields
\eqref{eq:comparable to distance function}, with constants having the
claimed dependence. It also follows that $v$ extends continuously to
the lateral boundary with value zero. Furthermore, the boundary
bounds and the local uniform convergence to $v_0$ imply
\[
v(\cdot,t)\to v_0
\quad\text{uniformly in }\ol\Omega
\quad\text{as }t\to0.
\]
Consequently,
\[
v\in C_{x,t}^{2,1}(\Omega\times(0,T])
\cap C(\overline{\Omega}\times[0,T]).
\]

Finally, we verify that
\[
v^p\in C([0,T];L^1(\Omega)).
\]
Let $t_j\to t_*\in[0,T]$. By the continuity and positivity of $v$,
\[
v(x,t_j)^p\to v(x,t_*)^p
\qquad\text{for every }x\in\Omega.
\]
Since $p<0$ and $v\geq c_Td$, we have
\[
0<v(x,t)^p\leq c_T^p d(x)^p
\qquad\text{in }\Omega\times[0,T].
\]
Because $p>-1$, the function $d^p$ belongs to $L^1(\Omega)$. The
dominated convergence theorem therefore gives
\[
\|v(\cdot,t_j)^p-v(\cdot,t_*)^p\|_{L^1(\Omega)}
\to0.
\]
Thus $v^p\in C([0,T];L^1(\Omega))$, and $v$ is an admissible solution
on $\Omega\times[0,T]$. Its uniqueness follows from
Theorem~\ref{thm:comparison}.
\end{proof}

\bigskip

\section{Gradient H\"older regularity for a degenerate equation}
\label{sec:C1alpha}

To establish the optimal regularity for the nonlinear equation, we
first derive boundary $C^{1,\gamma}$ estimates for the corresponding
linearized equations. Let $\bar v$ be a solution of the nonlinear
equation. By \eqref{eq:comparable to distance function},
\[
c d(x)\leq \bar v(x,t)\leq C d(x),
\]
and hence
\[
\bar v^{1-p}
=
d^{1-p}\left(\frac{\bar v}{d}\right)^{1-p}
\asymp d^{1-p}.
\]
Thus, after localizing and flattening the boundary, the principal part
of the linearized operator is reduced to
\[
\partial_t-x_n^{1-p}\bigl(a^{ij}D_{ij}+b^iD_i\bigr).
\]

The two-sided boundary estimate above is used only for the nonlinear
solution to identify the degeneracy of the linearized operator. We do
not impose the two-sided boundary growth condition on
solutions of the linear equation. Instead, linear boundary growth and
the corresponding $C^{1,\gamma}$ regularity will be derived from the
equation and the boundary data.

We therefore consider classical solutions $v\in C_{x,t}^{2,1}(Q_1^+)\cap C(\overline{Q}_1^+)$
of
\begin{equation}\label{original equation}
\begin{cases}
v_t-x_n^{1-p}(a^{ij}D_{ij}v+b^iD_i v)=x_n^q f
&\qquad \text{in } Q_1^+,\\
v=g
&\qquad \text{on } \partial_p Q_1^+,
\end{cases}
\end{equation}
where $p\in(-1,0)$, $q\in(-p,\infty)$, $f$ is bounded and continuous,
and $g$ is continuous and satisfies
\[
g=0 \quad \text{on} \quad
\partial_p Q_1^+\cap\{x_n=0\}.
\]
We assume that
\begin{equation}\label{eq:structure_conditions on Q1}
|a^{ij}|+|b^i|\leq \lambda^{-1},
\qquad
a^{ij}\xi_i\xi_j\geq \lambda|\xi|^2
\quad
\text{for all } \xi\in\mathbb R^n
\text{ and } (x,t)\in Q_1^+
\end{equation}
for some constant $0<\lambda<1$.

\begin{lem}[\cite{JTXZ}, Theorem 3.3]\label{lem:lipschitz estimate}
Let $p\in(-1,0)$, $q\in(-p,\infty)$, and
$f\in L^\infty(Q_1^+)$. Suppose that $v$ is a classical solution of
\eqref{original equation} satisfying $v=0$ on $\pa_p Q_1^+\cap\{x_n=0\}.$
Then
\[
|v(x,t)|
\leq
C\left(
\|g\|_{L^\infty(\pa_p Q_1^+)}
+\|f\|_{L^\infty(Q_1^+)}
\right)x_n
\qquad\text{for all }(x,t)\in\ol{Q}_{1/2}^+,
\]
where $C>0$ depends only on $n$, $p$, $q$, and $\lda$.
\end{lem}

\begin{lem}[Hopf Principle]\label{lem:hopf principle}
Let $p\in(-1,0)$, and suppose that $v$ is a nonnegative classical
solution of
\[
\begin{cases}
v_t-x_n^{1-p}\bigl(a^{ij}D_{ij}v+b^iD_i v\bigr)=0
&\text{in }Q_1^+,\\
v=0
&\text{on }\pa_p Q_1^+\cap\{x_n=0\}.
\end{cases}
\]
Then
\be\label{hopf principle}
v(x,t)
\geq
c\,v\left(
\frac{e_n}{2},
-\frac{1}{2}-\frac{1}{2^{p+2}}
\right)x_n
\qquad\text{for all }(x,t)\in\ol{Q}_{1/2}^+,
\ee
where $c\in(0,1)$ depends only on $n$, $p$, and $\lda$.
\end{lem}

\begin{proof}
Set
\[
B^*:=B_{1/4}\left(\frac{e_n}{2}\right),
\qquad
I^+:=\left(-\frac{1}{2^{p+1}},0\right],
\qquad
T_0:=\frac{1}{2^{p+1}},
\qquad
T_1:=\frac{1+T_0}{2}.
\]
Since the equation is uniformly parabolic away from $\{x_n=0\}$,
the interior Harnack inequality yields a constant $c_0>0$, depending
only on $n$, $p$, and $\lambda$, such that
\[
v\left(\frac{e_n}{2},-T_1\right)
\leq
\frac{1}{c_0}\inf_{B^*\times I^+}v.
\]
By continuity,
\[
v(x,t)\geq
c_0\,v\left(\frac{e_n}{2},-T_1\right)
\qquad
\text{for all }(x,t)\in\overline{B^*}\times\overline{I^+}.
\]

If $v(e_n/2,-T_1)=0$, then \eqref{hopf principle} follows immediately
from $v\geq0$. We may therefore assume that
$v(e_n/2,-T_1)>0$ and, by replacing $v$ with
\[
\frac{v(x,t)}{v(e_n/2,-T_1)},
\]
normalize $v(e_n/2,-T_1)=1.$

To construct a barrier near the flat boundary, set
\[
K:=\frac{1}{(p+2)(p+1)}.
\]
Fix $m_3>0$, choose $m_2\in(0,\lambda/2)$, and then take
$\varepsilon>0$ sufficiently small that
\[
K\varepsilon^{p+1}\leq\frac14m_2T_0.
\]
With
\[
\delta:=K\varepsilon^{p+1}+\frac{m_2T_0}{2},
\]
define, for $m_1>0$ to be chosen,
\[
\Phi(x,t)
:=
m_1\left[
Kx_n^{p+2}
+\bigl(m_2(t+T_0)-\delta\bigr)x_n
-m_3|x'|^2
\right].
\]
We compare $\Phi$ with $v$ in
\[
\Sigma
:=
\left\{
(x,t)\in Q_{1/2}^+:\ 0<x_n<\varepsilon
\right\}.
\]

We first verify that $\Phi$ is a subsolution. A direct computation gives
\[
\Phi_t=m_1m_2x_n,
\qquad
D_i\Phi=-2m_1m_3x_i,
\qquad
D_{ii}\Phi=-2m_1m_3
\quad (i=1,\ldots,n-1),
\]
and
\[
D_n\Phi
=
m_1K(p+2)x_n^{p+1}
+m_1\bigl(m_2(t+T_0)-\delta\bigr),
\qquad
D_{nn}\Phi=m_1x_n^p.
\]
All mixed derivatives vanish. Therefore, using
\eqref{eq:structure_conditions on Q1}, we obtain
\[
\begin{aligned}
&\Phi_t
-x_n^{1-p}
\bigl(a^{ij}D_{ij}\Phi+b^iD_i\Phi\bigr)\\
&=m_1\Biggl[
m_2x_n-a^{nn}x_n
+2m_3\left(\sum_{i=1}^{n-1}a^{ii}\right)x_n^{1-p}\\
&\hspace{2.2cm}
-b^n\left(
K(p+2)x_n^2
+\bigl(m_2(t+T_0)-\delta\bigr)x_n^{1-p}
\right)
+2m_3x_n^{1-p}\sum_{i=1}^{n-1}b^ix_i
\Biggr]\\
&\leq
m_1x_n\left[
m_2-\lambda
+4m_3(n-1)\lambda^{-1}x_n^{-p}
+\frac{\lambda^{-1}}{p+1}x_n
+\lambda^{-1}m_2T_0x_n^{-p}
\right].
\end{aligned}
\]
Since $p<0$, we have $x_n^{-p}\to0$ as $x_n\to0$.
After decreasing $\varepsilon$ if necessary, depending only on
$m_2$, $m_3$, $p$, and $\lambda$, we obtain
\[
m_2-\lambda
+4m_3(n-1)\lambda^{-1}x_n^{-p}
+\frac{\lambda^{-1}}{p+1}x_n
+\lambda^{-1}m_2T_0x_n^{-p}
\leq0
\]
whenever $0<x_n<\varepsilon$. Hence
\[
\Phi_t
-x_n^{1-p}
\bigl(a^{ij}D_{ij}\Phi+b^iD_i\Phi\bigr)
\leq0
\qquad\text{in }\Sigma.
\]

It remains to verify the boundary inequality. Let
\[
A:=
\left\{
x_n=\varepsilon,\quad
-\frac{T_0}{2}\leq t\leq0,\quad
|x'|\leq\frac14
\right\}.
\]
By the definition of $\delta$, and after decreasing $\varepsilon$ once
more if necessary,
\[
\Phi\leq0\leq v
\qquad\text{on }\partial_p\Sigma\setminus A.
\]
On the other hand, $A$ lies a positive distance away from the
degenerate boundary. The interior Harnack inequality
gives a constant $\widetilde c>0$, depending only on $n$, $p$, and
$\lambda$, such that
\[
v\geq\widetilde c
\qquad\text{on }A.
\]
Moreover,
\[
\Phi\leq m_1\frac{m_2T_0}{2}
\qquad\text{on }A.
\]
Choosing $m_1>0$ sufficiently small, depending only on $n$, $p$, and
$\lambda$, we obtain $\Phi\leq v$ on $A.$
Therefore, $
\Phi\leq v$ on $\partial_p\Sigma.$
The comparison principle now gives
\[
\Phi\leq v
\qquad\text{in }\Sigma.
\]

In particular, since
\[
m_2T_0-\delta
=
\frac{m_2T_0}{2}-K\varepsilon^{p+1}
\geq\frac{m_2T_0}{4},
\]
we have
\[
\begin{aligned}
v(0',x_n,0)
&\geq\Phi(0',x_n,0)\\
&=
m_1\left[
\bigl(m_2T_0-\delta\bigr)x_n
+Kx_n^{p+2}
\right]\\
&\geq
\frac{m_1m_2T_0}{4}x_n
\end{aligned}
\]
for $0<x_n<\varepsilon$.

Applying the same barrier argument after translations in the $x'$-
and $t$-directions gives the corresponding estimate throughout the
boundary strip $\{0<x_n<\varepsilon\}\cap\ol Q_{1/2}^+$. In the region
$x_n\geq\varepsilon$, the desired estimate follows from the interior
Harnack inequality. Restoring the original normalization, we conclude
that
\[
v(x,t)
\geq
c\,v\left(
\frac{e_n}{2},
-\frac{1}{2}-\frac{1}{2^{p+2}}
\right)x_n
\qquad\text{for all }(x,t)\in\ol Q_{1/2}^+,
\]
where $c\in(0,1)$ depends only on $n$, $p$, and $\lambda$. This proves
\eqref{hopf principle}.
\end{proof}

\begin{cor}
\label{cor:hopf principle inhomogeneous}
Let $p\in(-1,0)$, $q\in(-p,\infty)$, and
$f\in L^\infty(Q_1^+)$. Suppose that $v$ is a nonnegative classical
solution of
\[
\begin{cases}
v_t-x_n^{1-p}a^{ij}D_{ij}v=x_n^q f
&\qquad \text{in } Q_1^+,\\
v=0
&\qquad \text{on } \partial_p Q_1^+\cap\{x_n=0\}.
\end{cases}
\]
Then
\[
v(x,t)
\geq
\left[
c\,v\left(
\frac{e_n}{2},
-\frac{1}{2}-\frac{1}{2^{p+2}}
\right)
-C\|f\|_{L^\infty(Q_1^+)}
\right]x_n
\]
for every $(x,t)\in\overline Q_{1/2}^+$, where
$c=c(n,p,\lambda)\in(0,1)$ is the constant in
Lemma~\ref{lem:hopf principle}, and
$C=C(n,p,q,\lambda)>0$.
\end{cor}

\begin{proof}
Let
\[
P_0:=
\left(
\frac{e_n}{2},
-\frac{1}{2}-\frac{1}{2^{p+2}}
\right),
\]
and let $w$ solve
\[
\begin{cases}
w_t-x_n^{1-p}a^{ij}D_{ij}w=0
&\text{in }Q_1^+,\\
w=v
&\text{on }\pa_p Q_1^+.
\end{cases}
\]
Since $v\geq0$ on $\pa_pQ_1^+$, the maximum principle gives
$w\geq0$ in $Q_1^+$. Moreover, $w=0$ on
$\pa_pQ_1^+\cap\{x_n=0\}$.

The function $w-v$ satisfies
\[
(w-v)_t-x_n^{1-p}a^{ij}D_{ij}(w-v)
=-x_n^qf
\qquad\text{in }Q_1^+
\]
and vanishes on $\pa_pQ_1^+$. Hence, by
Lemma~\ref{lem:lipschitz estimate},
\[
|w-v|
\leq
C\|f\|_{L^\infty(Q_1^+)}x_n
\qquad\text{in }Q_{3/4}^+.
\]
Applying Lemma~\ref{lem:hopf principle} to $w$, we obtain
\[
w(x,t)\geq c\,w(P_0)x_n
\qquad\text{for all }(x,t)\in\ol Q_{1/2}^+.
\]
Since $P_0\in Q_{3/4}^+$ and its $x_n$-coordinate is $1/2$, the
preceding Lipschitz estimate gives
\[
w(P_0)
\geq
v(P_0)-C\|f\|_{L^\infty(Q_1^+)}.
\]
Therefore, for every $(x,t)\in\ol Q_{1/2}^+$,
\[
\begin{aligned}
v(x,t)
&\geq
w(x,t)-C\|f\|_{L^\infty(Q_1^+)}x_n\\
&\geq
c\,w(P_0)x_n
-C\|f\|_{L^\infty(Q_1^+)}x_n\\
&\geq
\left(
c\,v(P_0)
-C\|f\|_{L^\infty(Q_1^+)}
\right)x_n,
\end{aligned}
\]
where the constant $C>0$ has been enlarged in the last line. This
proves the desired estimate.
\end{proof}

\begin{thm}\label{thm:boundary 1 alpha regularity}
Let $p\in(-1,0)$, $q\in(-p,\infty)$, and
$f\in L^\infty(Q_1^+)$. Suppose that $v$ is a classical solution of
\[
\begin{cases}
v_t-x_n^{1-p}a^{ij}D_{ij}v=x_n^q f
&\qquad \text{in }Q_1^+,\\
v=0
&\qquad \text{on }\partial_p Q_1^+\cap\{x_n=0\}.
\end{cases}
\]
Then there exists $\gamma\in(0,p+q)$, depending only on
$n$, $p$, $q$, and $\lda$, with the following property: for every $(x_0,t_0)\in\partial_p Q_{1/2}^+\cap\{x_n=0\},$
there exists $a_{(x_0,t_0)}\in\R$ such that
\[
Dv(x_0,t_0)=a_{(x_0,t_0)}e_n\qquad \text{and}\qquad 
|a_{(x_0,t_0)}|
\leq
C\left(
\|v\|_{L^\infty(Q_1^+)}
+\|f\|_{L^\infty(Q_1^+)}
\right).
\]
Moreover, for every
$(x,t)\in Q_1^+\cap Q_{1/2}^+(x_0,t_0)$,
\[
\begin{aligned}
|v(x,t)-a_{(x_0,t_0)}x_n|
&\leq
C\left(
\|v\|_{L^\infty(Q_1^+)}
+\|f\|_{L^\infty(Q_1^+)}
\right)
\left(
|x-x_0|+|t-t_0|^{\frac1{p+1}}
\right)^{1+\gamma},
\end{aligned}
\]
where $C>0$ depends only on $n$, $p$, $q$, and $\lda$.
\end{thm}

\begin{proof}
Without loss of generality, assume that $(x_0,t_0)=(0,0)$,
\[
\|v\|_{L^\infty(Q_1^+)}
+\|f\|_{L^\infty(Q_1^+)}
\leq1,
\]
and consider only the case $q<1-p$. Indeed, if $q\geq1-p$, fix
$\widetilde q\in(-p,1-p)$ and write
\[
x_n^qf=x_n^{\widetilde q}\widetilde f,
\qquad
\widetilde f:=x_n^{q-\widetilde q}f.
\]
Then
\[
\|\widetilde f\|_{L^\infty(Q_1^+)}
\leq
\|f\|_{L^\infty(Q_1^+)}
\leq1.
\]

We claim that there exist two sequences
$\{a_k\}_{k=0}^\infty$ and $\{b_k\}_{k=0}^\infty$, with
$\{a_k\}$ nondecreasing and $\{b_k\}$ nonincreasing, such that, for
every $k\geq1$,
\begin{equation}\label{eq:ak and bk}
\begin{cases}
a_kx_n\leq v(x,t)\leq b_kx_n
&\text{for all }(x,t)\in Q_{1/2^k}^+,\\[1mm]
0\leq b_k-a_k\leq\mu2^{-k\gamma},
\end{cases}
\end{equation}
where $\mu>0$ and $\gamma\in(0,1)$ are universal constants depending
only on $n,p,q$, and $\lda$.

Assume that the claim holds. Then there exists $a\in\mathbb R$, with
$|a|\leq C$, such that
\[
\lim_{k\to\infty}a_k
=a
=\lim_{k\to\infty}b_k,
\qquad
a_1\leq a_2\leq\cdots\leq a
\leq\cdots\leq b_2\leq b_1.
\]
Thus, for every $k\geq1$ and
$(x,t)\in Q_{1/2^k}^+\setminus Q_{1/2^{k+1}}^+$,
\[
\begin{aligned}
|v(x,t)-ax_n|
&\leq
(b_k-a_k)x_n\leq
\mu2^{-k\gamma}x_n \leq
C\left(
|x|+|t|^{\frac1{p+1}}
\right)^\gamma x_n.
\end{aligned}
\]
Scaling back yields the desired estimates.

It remains to prove the claim. Let $0<c<1$ be the constant in
Lemma~\ref{lem:hopf principle}. Choose $\gamma\in(0,1)$ sufficiently
small so that
\[
2^\gamma<1+\frac{c}{8}
\qquad\text{and}\qquad
p+q>\gamma.
\]
By Lemma~\ref{lem:lipschitz estimate}, there exists $C_0>0$ such that
\[
-C_0x_n\leq v(x,t)\leq C_0x_n
\qquad
\text{in }\ol Q_{1/2}^+.
\]
Set
\[
a_0=-2C_0,\qquad a_1=-C_0,\qquad
b_0=2C_0,\qquad b_1=C_0,
\]
and
\[
\mu
=
C_02^{\gamma+4}
+2^{\gamma+3}c^{-1}C_1,
\]
where $C_1=C(n,p,q,\lambda)$ is the constant appearing in
Corollary~\ref{cor:hopf principle inhomogeneous}. Then \eqref{eq:ak and bk} holds for $k=1$. 

Suppose that
\eqref{eq:ak and bk} holds for some $k\geq1$, and set
$r_k=2^{-k}$. If
\[
b_k-a_k\leq\mu r_{k+1}^\gamma,
\]
we take $a_{k+1}=a_k,$ and $b_{k+1}=b_k,$
and the induction step is complete. We may therefore assume that
\[
b_k-a_k>\mu r_{k+1}^\gamma.
\]
At the point
\[
P_k:=
\left(
\frac{r_k}{2}e_n,
-\left(\frac12+\frac1{2^{p+2}}\right)r_k^{p+1}
\right),
\]
one of the following alternatives holds:
\[
v(P_k)\geq\frac{r_k}{4}(a_k+b_k)
\qquad\text{or}\qquad
v(P_k)\leq\frac{r_k}{4}(a_k+b_k).
\]

Suppose first that
\[
v(P_k)\geq\frac{r_k}{4}(a_k+b_k).
\]
Define
\[
w(x,t):=
\frac{v(r_kx,r_k^{p+1}t)-a_kr_kx_n}
{r_k^{1+\gamma}},
\]
and
\[
a_k^{ij}(x,t)
:=
a^{ij}(r_kx,r_k^{p+1}t),
\qquad
f_k(x,t)
:=
r_k^{p+q-\gamma}f(r_kx,r_k^{p+1}t).
\]
Then $w$ satisfies
\[
\begin{cases}
w_t-x_n^{1-p}a_k^{ij}D_{ij}w=x_n^qf_k
&\text{in }Q_1^+,\\
w=0
&\text{on }\partial_pQ_1^+\cap\{x_n=0\},
\end{cases}
\]
and the induction hypothesis gives
\[
0\leq w\leq\mu
\qquad
\text{in }Q_1^+.
\]
Moreover, Corollary~\ref{cor:hopf principle inhomogeneous} and
$\|f_k\|_{L^\infty(Q_1^+)}\leq1$ yield
\begin{align*}
w(x,t)
&\geq
\left[
c\,w\left(
\frac{e_n}{2},
-\frac12-\frac1{2^{p+2}}
\right)
-C_1\|f_k\|_{L^\infty(Q_1^+)}
\right]x_n\\
&\geq
\left(
\frac{c(b_k-a_k)}{4r_k^\gamma}
-C_1
\right)x_n\\
&\geq
\left(
2^{-\gamma-2}c\mu-C_1
\right)x_n\\
&\geq
2^{-\gamma-3}c\mu x_n.
\end{align*}
Scaling back, we obtain
\[
v(x,t)\geq
\left(
a_k+2^{-\gamma-3}c\mu r_k^\gamma
\right)x_n
\qquad
\text{in }Q_{r_{k+1}}^+.
\]
Set
\[
a_{k+1}
:=
a_k+2^{-\gamma-3}c\mu r_k^\gamma,
\qquad
b_{k+1}:=b_k.
\]
Then
\[
b_{k+1}-a_{k+1}
\leq
\left(
1-2^{-\gamma-3}c
\right)\mu r_k^\gamma
\leq
\mu r_{k+1}^\gamma,
\]
where the last inequality follows from the choice of $\gamma$.

The second alternative is handled in the same way by applying the
above argument to
\[
\frac{
b_kr_kx_n-v(r_kx,r_k^{p+1}t)
}
{r_k^{1+\gamma}}.
\]
Therefore, \eqref{eq:ak and bk} holds for $k+1$, and the induction is
complete.
\end{proof}

\begin{thm}\label{thm:C1alpha of nonlinear equation in Q1}
Let $p\in(-1,0)$, $q\in(-p,\infty)$, and let
$\gamma\in(0,\min\{p+q,1\})$ be as in
Theorem~\ref{thm:boundary 1 alpha regularity}. Suppose that
\[
a^{ij},b^i\in C_{x,t}^{\gamma,\gamma/2}(Q_1^+)
\]
satisfy \eqref{eq:structure_conditions on Q1}, and let
$f\in L^\infty(Q_1^+)$. If $v$ is a classical solution of
\[
\begin{cases}
v_t-x_n^{1-p}\left(a^{ij}D_{ij}v+b^iD_iv\right)=x_n^qf
&\text{in }Q_1^+,\\
v=0
&\text{on }\pa_p Q_1^+\cap\{x_n=0\},
\end{cases}
\]
then $v\in C_{x,t}^{1+\gamma,\frac{1+\gamma}{2}}(Q_{1/2}^+)$
and
\[
\|v\|_{C_{x,t}^{1+\gamma,\frac{1+\gamma}{2}}(Q_{1/2}^+)}
\leq
C\left(
\|v\|_{L^\infty(Q_1^+)}
+
\|f\|_{L^\infty(Q_1^+)}
\right),
\]
where $C>0$ depends only on $n,p,q,\gamma,\lda$, and the
$C_{x,t}^{\gamma,\gamma/2}$ norms of $a^{ij}$ and $b^i$.
\end{thm}

\begin{proof}
By the same normalization as in
Theorem~\ref{thm:boundary 1 alpha regularity}, we may assume that
\[
\|v\|_{L^\infty(Q_1^+)}\leq1,
\qquad
\|f\|_{L^\infty(Q_1^+)}\leq1.
\]

We first show that $Dv$ is bounded. For $0<r<1/2$, define
\[
v_r(x,t):=v(rx,r^{p+1}t).
\]
Then $v_r$ satisfies
\[
(v_r)_t-x_n^{1-p}
\left(a_r^{ij}D_{ij}v_r+b_r^iD_iv_r\right)
=
x_n^qr^{p+1+q}f(rx,r^{p+1}t)
\quad\text{in }Q_2^+,
\]
where
\[
a_r^{ij}(x,t):=a^{ij}(rx,r^{p+1}t),
\qquad
b_r^i(x,t):=rb^i(rx,r^{p+1}t).
\]
Since this equation is uniformly parabolic in
$Q_2^+\cap\{x_n\geq\delta\}$, standard interior regularity estimates
give
\[
|Dv_r(x,t)|
\leq
C\left(
\|v_r\|_{L^\infty(Q_2^+)}
+r^{p+q+1}\|f\|_{L^\infty(Q_1^+)}
\right)
\quad\text{in }Q_{1/2}(0',1,0).
\]
Scaling back, we obtain
\[
|Dv(x,t)|
\leq
C\left(
\frac{\|v\|_{L^\infty(Q_{2r}^+)}}{r}
+r^{p+q}\|f\|_{L^\infty(Q_1^+)}
\right)
\quad\text{in }Q_{r/2}(0',r,0).
\]
By Lemma~\ref{lem:lipschitz estimate}, 
$\|v\|_{L^\infty(Q_{2r}^+)}\leq Cr,$
and hence
\[
|Dv(x,t)|
\leq
C\left(
1+\|f\|_{L^\infty(Q_1^+)}
\right)
\leq C.
\]
Translations in the $x'$- and $t$-directions then show that $Dv$ is
bounded in $Q_{1/2}^+$.

Since $Dv$ is bounded, choose $\hat q$ such that
\[
\gamma-p<\hat q<\min\{q,1-p\}.
\]
We may rewrite the equation as
\[
v_t-x_n^{1-p}a^{ij}D_{ij}v
=
x_n^{\hat q}
\left(
x_n^{q-\hat q}f
+x_n^{1-p-\hat q}b^iD_iv
\right).
\]
Since $q-\hat q>0$, $1-p-\hat q>0$, and $Dv$ is bounded, the
right-hand side is bounded. Moreover, $\gamma<p+\hat q$. Therefore,
by Theorem~\ref{thm:boundary 1 alpha regularity}, for every point
\[
z_0=(x_0',0,t_0)
\in\pa_p Q_{1/2}^+\cap\{x_n=0\},
\]
there exists a linear function
\[
L_{z_0}(x):=a_{z_0}x_n
\]
such that
\[
|v(x,t)-L_{z_0}(x)|
\leq
C\left(
|x-x_0|+|t-t_0|^{\frac1{p+1}}
\right)^{1+\gamma},
\]
with $|a_{z_0}|\leq C$.

We next establish the
$C_{x,t}^{1+\gamma,\frac{1+\gamma}{2}}$ estimate away from
$\{x_n=0\}$. For $0<r<1/2$, define
\[
w_{r,z_0}(x,t)
:=
\frac{
(v-L_{z_0})\bigl(z_0+(rx,r^{p+1}t)\bigr)
}{r^{p+1}},
\qquad
(x,t)\in Q_1^+.
\]
Then $w_{r,z_0}$ satisfies
\[
(w_{r,z_0})_t-x_n^{1-p}
\left(
a_r^{ij}D_{ij}w_{r,z_0}
+b_r^iD_iw_{r,z_0}
\right)
=
x_n^q\wt f,
\]
where
\[
a_r^{ij}(x,t)
:=
a^{ij}\bigl(z_0+(rx,r^{p+1}t)\bigr),
\qquad
b_r^i(x,t)
:=
rb^i\bigl(z_0+(rx,r^{p+1}t)\bigr),
\]
and
\[
\wt f(x,t)
:=
r^qf\bigl(z_0+(rx,r^{p+1}t)\bigr)
+a_{z_0}r^{1-p} x_n^{1-p-q}
b^n\bigl(z_0+(rx,r^{p+1}t)\bigr).
\]
The interior $C^{1,\gamma}$ estimate gives
\begin{align*}
\|w_{r,z_0}\|_
{C_{x,t}^{1+\gamma,\frac{1+\gamma}{2}}
(Q_{1/4}(0',1/2,0))}
&\leq
C\left(
\|w_{r,z_0}\|_{L^\infty(Q_{3/8}(0',1/2,0))}
+
\|\wt f\|_{L^\infty(Q_{3/8}(0',1/2,0))}
\right)\\
&\leq
C\left(
\frac{1}{r^{p+1}}
\|v-L_{z_0}\|_{L^\infty(Q_r^+(z_0))}
+
\|\wt f\|_{L^\infty(Q_{3/8}(0',1/2,0))}
\right)\\
&\leq
C\left(
r^{\gamma-p}+r^q+r^{1-p}
\right)\\
&\leq
Cr^{\gamma-p},
\end{align*}
where we used $\gamma<p+q$.

By the definition of $w_{r,z_0}$, for $i=1,\ldots,n$,
\[
D_iv-\delta_i^na_{z_0}=r^pD_iw_{r,z_0}.
\]
Scaling back, we obtain
\[
\begin{aligned}
\bigl[D_i v\bigr]_{C_x^\gamma(Q_{r/4}(x_0',r/2,t_0))}
&=
r^{p-\gamma}
[D_iw_{r,z_0}]_{C_x^\gamma(Q_{1/4}(x_0',1/2,0))}\leq C
\end{aligned}
\]
and
\[
\begin{aligned}
\bigl[D_iv\bigr]_{C_t^{\gamma/2}(Q_{r/4}(x_0',r/2,t_0))}
&=
r^{p-\frac{(p+1)\gamma}{2}}
[D_iw_{r,z_0}]_{C_t^{\gamma/2}(Q_{1/4}(x_0',1/2,0))}
\leq
Cr^{\frac{\gamma(1-p)}2}
\leq C.
\end{aligned}
\]
Moreover, since
\[
v-L_{z_0}=r^{p+1}w_{r,z_0},
\]
we have
\[
\begin{aligned}
\bigl[v\bigr]_{C_t^{\frac{1+\gamma}{2}}
(Q_{r/4}(x_0',r/2,t_0))}
&=
r^{p+1-\frac{(p+1)(1+\gamma)}{2}}
[w_{r,z_0}]_{C_t^{\frac{1+\gamma}{2}}
(Q_{1/4}(x_0',1/2,0))}\leq
Cr^{\frac{(1-p)(1+\gamma)}2}
\leq C.
\end{aligned}
\]

The preceding scaled interior estimate also implies that, for every
$z_0=(x_0',h,t_0)\in Q_{1/2}^+$, with
$\overline z_0=(x_0',0,t_0)$,
\begin{equation}\label{eq:gradient boundary approximation}
|Dv(z_0)-a_{\overline z_0}e_n|
\leq Ch^\gamma.
\end{equation}
Indeed, when $h$ is sufficiently small, this follows by taking
$r=2h$ in the estimate for $w_{r,\overline z_0}$. When $h$ is bounded
away from zero, it follows from
\[
|Dv|+|a_{\overline z_0}|\leq C.
\]

We also claim that the boundary slopes are H\"older continuous. More
precisely, if
\[
\xi=(\xi',0,\tau),
\qquad
\eta=(\eta',0,\sigma),
\]
then
\begin{equation}\label{eq:boundary slopes holder}
|a_\xi-a_\eta|
\leq
C\left(
|\xi'-\eta'|+|\tau-\sigma|^{\frac12}
\right)^\gamma.
\end{equation}
Indeed, set
\[
\delta
:=
|\xi'-\eta'|
+
|\tau-\sigma|^{\frac1{p+1}}.
\]
Assuming that $\tau\geq\sigma$, we evaluate the boundary expansions
at
\[
P=(\xi',\delta,\sigma)
\]
to obtain
\[
\begin{aligned}
\delta|a_\xi-a_\eta|
&\leq
|v(P)-a_\xi\delta|
+
|v(P)-a_\eta\delta|
\leq
C\delta^{1+\gamma}.
\end{aligned}
\]
This proves \eqref{eq:boundary slopes holder}, since
\[
|\tau-\sigma|^{1/(p+1)}
\leq
|\tau-\sigma|^{1/2}.
\]
When $\delta$ is bounded away from zero, the same estimate follows
from
\[
|a_\xi|+|a_\eta|\leq C.
\]

We now combine the boundary and interior estimates. Let
\[
z_0=(x,t),
\qquad
z_1=(y,s)\in Q_{1/2}^+,
\]
and set
\[
\rho:=|x-y|+|t-s|^{\frac12}.
\]
Without loss of generality, assume that
\[
h:=x_n\leq y_n = :k.
\]
If $\rho<h/4$, the interior estimate gives
\[
|Dv(z_0)-Dv(z_1)|
\leq
C\rho^\gamma.
\]
If $\rho\geq h/4$, then
\[
h\leq4\rho,
\qquad
k\leq h+|x-y|\leq5\rho.
\]
Writing
\[
\overline z_0=(x',0,t),
\qquad
\overline z_1=(y',0,s),
\]
and using \eqref{eq:gradient boundary approximation} and
\eqref{eq:boundary slopes holder}, we obtain
\[
\begin{aligned}
|Dv(z_0)-Dv(z_1)|
&\leq
|Dv(z_0)-a_{\overline z_0}e_n|
+
|a_{\overline z_0}-a_{\overline z_1}|
+
|a_{\overline z_1}e_n-Dv(z_1)|\\
&\leq
C\left(
h^\gamma+\rho^\gamma+k^\gamma
\right)\\
&\leq
C\rho^\gamma.
\end{aligned}
\]

It remains to establish the H\"older continuity of $v$ in time. Fix
$(x,t),(x,s)\in Q_{1/2}^+$ and set
\[
h:=x_n,
\qquad
\tau:=|t-s|^{\frac12}.
\]
If $\tau<h/4$, the preceding local interior estimate gives
\[
|v(x,t)-v(x,s)|
\leq
C|t-s|^{\frac{1+\gamma}{2}}.
\]
If $\tau\geq h/4$, then, using the boundary condition
$
v(x',0,t)=v(x',0,s)=0
$ 
and the estimate
\[
[Dv]_{C_t^{\gamma/2}(Q_{1/2}^+)}
\leq C,
\]
we obtain
\[
\begin{aligned}
|v(x,t)-v(x,s)|
&=
\left|
\int_0^h
\left(
D_nv(x',\xi,t)-D_nv(x',\xi,s)
\right)\,d\xi
\right|\\
&\leq
Ch|t-s|^{\frac{\gamma}{2}}\\
&\leq
C|t-s|^{\frac{1+\gamma}{2}}.
\end{aligned}
\]
Combining the two cases yields
$
[v]_{C_t^{\frac{1+\gamma}{2}}(Q_{1/2}^+)}
\leq C.
$
Together with the H\"older estimate for $Dv$, this completes
\[
v\in
C_{x,t}^{1+\gamma,\frac{1+\gamma}{2}}(Q_{1/2}^+).
\]
Undoing the normalization gives the desired estimate and completes
the proof.
\end{proof}

\begin{thm}\label{thm:C1alpha of nonlinear equation}
Let $\om\subset\mathbb R^n$ be a smooth bounded domain, let $T>0$, and
let $p\in(-1,0)$. Suppose that $v_0\in C(\ol\om)$ satisfies
\[
c_0d(x)\leq v_0(x)\leq C_0d(x)
\qquad
\text{in }\om
\]
for some constants $0<c_0<C_0<\infty$. Then
\eqref{eq:equation for v} admits a unique admissible solution $v$.
Moreover, there exists $\gamma\in(0,1)$ such that
$
v\in
C_{x,t}^{1+\gamma,\frac{1+\gamma}{2}}
(\ol\om\times [T/2,T])
$
with
\[
\|v\|_{C_{x,t}^{1+\gamma,\frac{1+\gamma}{2}}
(\ol\om\times\left[\frac{T}{2},T\right])}
\leq
C\|v\|_{L^\infty(\om\times(0,T])},
\]
where $\gamma$ and $C>0$ depend only on
$n,p,\|\om\|_{C^3},T,c_0$, and $C_0$.
\end{thm}

\begin{proof}
Existence and uniqueness follow from
Theorems~\ref{thm:comparison} and~\ref{thm:existence}. Moreover, the
proof of Theorem~\ref{thm:existence} yields constants
$0<c_T\leq C_T<\infty$ such that
\[
c_Td(x)\leq v(x,t)\leq C_Td(x)
\qquad
\text{in }\om\times[0,T].
\]
By decomposing the domain, flattening the boundary, and then applying
a sliding transformation, it suffices to consider the equation
\[
v_t-v^{1-p}
\left(
a^{ij}D_{ij}v+b^iD_iv
\right)
=0
\qquad
\text{in }Q_1^+,
\]
where 
$a^{ij},b^i\in
C_{x,t}^{\gamma_0,\gamma_0/2}(Q_1^+)$
for some $\gamma_0\in (0,1)$.

Since $v$ is comparable to $x_n$ near the boundary, the same argument
as in the first part of the proof of
Theorem~\ref{thm:C1alpha of nonlinear equation in Q1} shows that $v$
is Lipschitz continuous near the boundary and
\[
\|Dv\|_{L^\infty(Q_1^+)}
\leq
C\|v\|_{L^\infty(Q_1^+)}.
\]
The equation
\[
v_t-v^{1-p}a^{ij}D_{ij}v
=
v^{1-p}b^iD_iv
\qquad
\text{in }Q_1^+
\]
can therefore be written as
\[
v_t-x_n^{1-p}
\left(\frac{v}{x_n}\right)^{1-p}
a^{ij}D_{ij}v
=
x_n^{1-p}
\left(\frac{v}{x_n}\right)^{1-p}
b^iD_iv.
\]
Applying Theorem~\ref{thm:boundary 1 alpha regularity} with
\[
q=1-p
\qquad \text{and}\qquad
f=
\left(\frac{v}{x_n}\right)^{1-p}b^iD_iv,
\]
we find that, for every
$
(x_0,t_0)\in
\pa_pQ_{1/2}^+\cap\{x_n=0\},
$
there exists a constant $a_{(x_0,t_0)}$ such that
\[
|a_{(x_0,t_0)}|
\leq
C\|v\|_{L^\infty(Q_1^+)}
\]
and
\[
|v(x,t)-a_{(x_0,t_0)}x_n|
\leq
C\|v\|_{L^\infty(Q_1^+)}
\left(
|x-x_0|+|t-t_0|^{\frac1{p+1}}
\right)^{1+\gamma_1}
\]
for all
$
(x,t)\in
Q_1^+\cap Q_{1/2}^+(x_0,t_0).
$

We now use a standard scaling argument to establish regularity near
the boundary. Assume that $(x_0,t_0)=(0,0)$ and set
$
M:=\|v\|_{L^\infty(Q_1^+)}.
$
For $0<r<1/2$, define
\[
v_r(x,t):=
\frac{v(rx,r^{p+1}t)}{r}
\qquad
\text{in }Q_1^+.
\]
Then $v_r$ satisfies
\[
(v_r)_t-v_r^{1-p}
\left(
a_r^{ij}D_{ij}v_r+b_r^iD_iv_r
\right)
=0
\qquad
\text{in }Q_1^+,
\]
where
\[
a_r^{ij}(x,t)
:=
a^{ij}(rx,r^{p+1}t),
\qquad
b_r^i(x,t)
:=
rb^i(rx,r^{p+1}t).
\]
Since $v$ is comparable to $x_n$, we also have
\[
c_Tx_n\leq v_r(x,t)\leq C_Tx_n.
\]
Moreover, the classical interior H\"older estimate gives
\[
\|v_r\|_
{C_{x,t}^{\gamma_2,\frac{\gamma_2}{2}}
(Q_{1/3}(0',1/2,0))}
\leq C
\]
for some $\gamma_2\in(0,1)$, where $C$ is independent of $r$.

By the boundary regularity of $v$ at $(0,0)$, there exists a constant
$a$ such that, for $(x,t)\in Q_r^+$,
\[
\begin{aligned}
|v(x,t)-ax_n|
&\leq
CM\left(
|x|+|t|^{\frac1{p+1}}
\right)^{1+\gamma_1}\leq
CMr^{1+\gamma_1}.
\end{aligned}
\]
Therefore,
\[
|v_r(x,t)-ax_n|
\leq
CMr^{\gamma_1}
\qquad
\text{in }Q_1^+.
\]
Decreasing $\gamma$ if necessary, we may assume that
$
0<\gamma<
\min\{\gamma_0,\gamma_1,\gamma_2\}.
$
Let
\[
w_r(x,t):=
\frac{v_r(x,t)-ax_n}{r^\gamma}
\qquad
\text{in }Q_1^+.
\]
Then $|w_r|\leq CM$, and
\[
(w_r)_t-v_r^{1-p}
\left(
a_r^{ij}D_{ij}w_r+b_r^iD_iw_r
\right)
=
ar^{1-\gamma}v_r^{1-p}
b^n(rx,r^{p+1}t)
\qquad
\text{in }Q_1^+.
\]
Since
$
v_r\in
C_{x,t}^{\gamma_2,\frac{\gamma_2}{2}}
(Q_{1/3}(0',1/2,0))
$
and
$
a^{ij},b^i\in
C_{x,t}^{\gamma_0,\frac{\gamma_0}{2}}(Q_1^+),
$
the coefficients and the right-hand side in the equation for $w_r$
have uniformly bounded $C_{x,t}^{\gamma,\gamma/2}$ norms in
$Q_{1/3}(0',1/2,0)$. The interior Schauder estimates therefore give
$
w_r\in
C_{x,t}^{2+\gamma,\frac{2+\gamma}{2}}
(Q_{1/4}(0',1/2,0))
$
and
\[
\|w_r\|_
{C_{x,t}^{2+\gamma,\frac{2+\gamma}{2}}
(Q_{1/4}(0',1/2,0))}
\leq
CM.
\]

Returning to the original variables, we have
\[
v(x,t)
=
ax_n+
r^{1+\gamma}
w_r\left(
\frac{x}{r},
\frac{t}{r^{p+1}}
\right).
\]
Hence,
$
v\in
C_{x,t}^{1+\gamma,\frac{1+\gamma}{2}}
\left(
Q_{r/4}(0',r/2,0)
\right).
$
Moreover, for
$
(x,t),(y,s)\in Q_{r/4}(0',r/2,0),
$
we have
\begin{align*}
|Dv(x,t)-Dv(y,s)|
&=
r^\gamma
\left|
Dw_r\left(
\frac{x}{r},
\frac{t}{r^{p+1}}
\right)
-
Dw_r\left(
\frac{y}{r},
\frac{s}{r^{p+1}}
\right)
\right|\\
&\leq
CMr^\gamma
\left(
\frac{|x-y|}{r}
+
\left|
\frac{t-s}{r^{p+1}}
\right|^{\frac12}
\right)^\gamma\\
&\leq
CM\left(
|x-y|+|t-s|^{\frac12}
\right)^\gamma,
\end{align*}
and
\begin{align*}
|v(x,t)-v(x,s)|
&=
r^{1+\gamma}
\left|
w_r\left(
\frac{x}{r},
\frac{t}{r^{p+1}}
\right)
-
w_r\left(
\frac{x}{r},
\frac{s}{r^{p+1}}
\right)
\right|\\
&\leq
CMr^{1+\gamma}
\left|
\frac{t-s}{r^{p+1}}
\right|^{\frac{1+\gamma}{2}}\\
&\leq
CMr^{\frac{(1-p)(1+\gamma)}{2}}
|t-s|^{\frac{1+\gamma}{2}}\\
&\leq
CM|t-s|^{\frac{1+\gamma}{2}}.
\end{align*}
By translation in $x'$ and $t$, the preceding estimates hold for
every
$
(x_0,t_0)\in
\pa_pQ_{1/2}^+\cap\{x_n=0\}.
$

Repeating the argument in the proof of
Theorem~\ref{thm:C1alpha of nonlinear equation in Q1}, we obtain
$
v\in
C_{x,t}^{1+\gamma,\frac{1+\gamma}{2}}(Q_{1/2}^+)
$
and
\[
\|v\|_
{C_{x,t}^{1+\gamma,\frac{1+\gamma}{2}}(Q_{1/2}^+)}
\leq
CM.
\]
Finally, a finite covering argument, together with the interior
Schauder estimates for uniformly parabolic equations, yields
$
v\in
C_{x,t}^{1+\gamma,\frac{1+\gamma}{2}}
(\ol\om\times [T/2,T])
$
and
\[
\|v\|_
{C_{x,t}^{1+\gamma,\frac{1+\gamma}{2}}
(\ol\om\times\left[\frac{T}{2},T\right])}
\leq
C\|v\|_{L^\infty(\om\times(0,T])}.
\]
\end{proof}

\section{All time regularity and large time asymptotics}\label{sec:regularity}

Let $\omega\in C^\infty(\ol\om)$ be comparable to the distance
function $d(x)$. More precisely, assume that there exists
$\lambda\in(0,1]$ such that
\begin{equation}\label{eq:assumptions of omega}
\lambda d(x)\leq\omega(x)\leq\lambda^{-1}d(x),
\qquad
\|\omega\|_{C^3(\ol\om)}\leq1.
\end{equation}
Such a function exists. For example, one may take a suitably
normalized positive first eigenfunction of $-\Delta$ in $\om$ with
zero Dirichlet boundary condition.

Since the solution of \eqref{eq:equation for v} is comparable to the distance
function, the corresponding linearized equation takes the form
\begin{equation}\label{eq:linearized equation}
\pa_tu
-\omega(x)^{1-p}
\left[
\sum_{i,j=1}^n a^{ij}(x,t)u_{ij}
+\sum_{i=1}^n b^i(x,t)u_i
\right]
+c(x,t)u
=
\omega(x)^qf(x,t)
\end{equation}
in $\om\times(0,T]$, where the coefficient matrix
$\{a^{ij}\}_{i,j=1}^n$ is symmetric and
\begin{equation}\label{eq:structure_conditions}
|a^{ij}|+|b^i|+|c|+|f|\leq\lambda^{-1},
\qquad
a^{ij}\xi_i\xi_j\geq\lambda|\xi|^2
\quad\text{for all } \xi\in\mathbb R^n,
\end{equation}
in $\om\times(0,T]$.

Following \cite{daskalopoulos1998regularity, kim2009smooth, JTXZ}, we introduce the weighted H\"older and Schauder
spaces associated with the intrinsic metric $s$. Let $s$ be the
Riemannian metric on $\mathbb R_+^n$ defined by
\[
ds^2
=
\frac{dx_1^2+\cdots+dx_n^2}{2x_n^\alpha},
\]
and denote the induced geodesic distance by $s(x,y)$. This distance is
comparable to
\[
\ol s(x,y)
:=
\frac{|x-y|}
{x_n^{\frac{\alpha}{2}}
+y_n^{\frac{\alpha}{2}}
+|x-y|^{\frac{\alpha}{2}}},
\]
in the sense that
\[
c(n,\alpha)\ol s(x,y)
\leq
s(x,y)
\leq
C(n,\alpha)\ol s(x,y).
\]
The corresponding parabolic distance on
$\mathbb R_+^n\times\mathbb R$ is defined by
\[
s\bigl((x,t),(y,\tau)\bigr)
:=
\max\left\{
s(x,y),\,|t-\tau|^{1/2}
\right\}.
\]

On $\om$, the distance $s_\om$ is defined near the boundary by pulling
back $s$ through local boundary-flattening maps and is chosen to be
equivalent to the Euclidean distance away from the boundary. The
resulting distance is comparable to
\[
\ol s_\om(x,y)
:=
\frac{|x-y|}
{d(x)^{\frac{\alpha}{2}}
+d(y)^{\frac{\alpha}{2}}
+|x-y|^{\frac{\alpha}{2}}}.
\]
The associated parabolic distance is defined by
\[
s_\om\bigl((x,t),(y,\tau)\bigr)
:=
\max\left\{
s_\om(x,y),\,|t-\tau|^{1/2}
\right\}.
\]
By abuse of notation, we continue to write $s$ and $\ol s$ for
$s_\om$ and $\ol s_\om$, respectively.

The weighted H\"older norm associated with $s$ is defined by
\[
\|w\|_{\C^\gamma(\om\times[0,T])}
:=
\|w\|_{L^\infty(\om\times[0,T])}
+
[w]_{\C^\gamma(\om\times[0,T])},
\]
where
\[
[w]_{\C^\gamma(\om\times[0,T])}
:=
\sup_{\substack{P_1,P_2\in\om\times[0,T]\\P_1\neq P_2}}
\frac{|w(P_1)-w(P_2)|}{s(P_1,P_2)^\gamma}.
\]

Let $p\in(-1,0)$, $q\in(-p,(1-p)/2]$, and
\[
0<\gamma<\frac{2(p+q)}{p+1}.
\]
We define the weighted Schauder norm by
\[
\begin{aligned}
\|w\|_{\C^{2+\gamma}(\om\times(0,T])}
:={}&
\|w\|_{\C^\gamma(\om\times(0,T])}
+\sum_{i=1}^n
\|w_{x_i}\|_{\C^\gamma(\om\times(0,T])} \\
&+
\|d^{-q}w_t\|_{\C^\gamma(\om\times(0,T])}
+\sum_{1\leq i\leq j\leq n}
\left\|
d^{1-p-q}w_{x_ix_j}
\right\|_{\C^\gamma(\om\times(0,T])}.
\end{aligned}
\]

Notice that $s_\om$ is equivalent to the Euclidean distance away from the
boundary, while near the boundary,
\[
|x-y|\leq C s_\om(x,y).
\]
Consequently,
\[
C^\gamma(\om)\subset \C^\gamma(\om).
\]

For convenience, we recall the  main theorem  in \cite{JTXZ} and some useful lemmas.
\begin{thm}[\cite{JTXZ}, Theorem 1.1]\label{thm1.1}
Let $\Omega \subset \mathbb{R}^n$ be a bounded domain with smooth boundary. Let
\[
p \in (-1,0), \quad q \in (-p, \frac{1-p}{2}]\quad \text{and} \quad  \gamma^* := \frac{2p+2q}{p+1}.
\]
Let $T>0$.  Suppose for some  exponent $\gamma \in (0, \gamma^*)$, the coefficients satisfy \eqref{eq:assumptions of omega}, \eqref{eq:structure_conditions} and 
\[
\|a^{ij}\|_{\C^\gamma(\Omega\times [0,T])} + \|b^i\|_{\C^\gamma(\Omega\times [0,T])} + \|c\|_{\C^\gamma(\Omega\times [0,T])} \le \lambda^{-1}.
\]
Assume also that $f \in \C^\gamma(\Omega \times [0,T])$. Then, for any $g_0\in C_0(\ol{\om})$, there exists a unique classical solution $v\in \C^{2+\gamma}(\Omega\times (0,T])\cap C(\overline\Omega\times[0,T])$ to \eqref{eq:linearized equation} satisfying $v=0$ on $\pa\om\times [0,T]$ and $v(x,0)=g_0(x)$ in $\om$.
Moreover, the following estimates hold:
\begin{enumerate}
\item[(i)] For any $t_0 \in (0,T]$, there exists $C_{T,t_0} > 0$, which depends only on $n$, $p$, $q$, $\gamma$, $\lambda$, $t_0$, $T$, and $\|\partial\Omega\|_{C^{2,(p+1)\gamma/2}}$, such that
\[
\|v\|_{\C^{2+\gamma}(\Omega\times (t_0,T])}
\le C_{T,t_0}\left( \|g_0\|_{L^{\infty}(\om)}+\|f\|_{\C^\gamma(\Omega\times [0,T])}\right).
\] 
Moreover,  the quantity $\omega^{-q}v_t$ and $\omega^{1-p-q}D^2 v$ are H\"older continuous up to the boundary, and their boundary traces satisfy
\begin{equation}\label{eq:compatibility condition}
\omega^{-q}v_t=0,\quad   \omega^{1-p-q}D^2 v =-\frac{f}{a^{\nu \nu }} \nu \otimes \nu \quad \text{on } \partial\Omega\times(0,T],
\end{equation}
where $\nu$ denotes the unit inner normal vector field on $\partial\Omega$  and
$a^{\nu\nu}:=a^{ij}\nu_i\nu_j$.
\item [(ii)] 
If, in addition,  $g_0 \in \C^{2+\gamma}(\Omega)$ and satisfies the compatibility condition
\begin{equation}\label{eq:compatibility condition t=0}
\omega^{1-p-q} a^{ij}D_{ij}g_{0}=-f  \quad \text{on }\pa\om\times \{t=0\},
\end{equation}
then we have $v\in \C^{2+\gamma}(\overline{\Omega}\times [0,T])$, and the following estimate holds:
\[
\|v\|_{\C^{2+\gamma}(\Omega\times [0,T])}
\le C_T\left( \|g_0\|_{\C^{2+\gamma}(\om)}+\|f\|_{\C^\gamma(\Omega\times [0,T])}\right),
\] 
where $C_{T} > 0$, depending only on $n$, $p$, $q$, $\gamma$, $\lambda$, $\|\partial\Omega\|_{C^{2,(p+1)\gamma/2}}$, and $T$.
\end{enumerate} 
\end{thm}

\begin{thm}\label{thm:cut off estimate}
Let $p,q,\gamma,\gamma^*,a^{ij},b^i,c$, and $f$ be as in
Theorem~\ref{thm1.1}. Suppose that
$
v\in\C^{2+\gamma}(\ol\om\times[0,T])
$
is a solution of
\[
\pa_tv
-\omega(x)^{1-p}
\left[
\sum_{i,j=1}^n a^{ij}(x,t)v_{ij}
+\sum_{i=1}^n b^i(x,t)v_i
\right]
+c(x,t)v
=
\omega(x)^qf(x,t)
\quad
\text{in }\om\times[0,T],
\]
satisfying
\[
v=0
\qquad
\text{on }\pa\om\times[0,T].
\]
Then, for every $r>0$, there exists a constant $C>0$, depending only
on $n,p,q,r,\gamma,\lambda$,
$\|\partial\Omega\|_{C^{2,(p+1)\gamma/2}}$, and $T$, such that
\[
\|v\|_{\C^{2+\gamma}
(\ol\om\times\left[\frac{T}{2},T\right])}
\leq
C\left(
\|v\|_{L^r(\ol\om\times[0,T])}
+
\|f\|_{\C^\gamma(\ol\om\times[0,T])}
\right).
\]
\end{thm}

\begin{proof}
Let $0<\tau<s\leq T$, and choose a cutoff function $\eta(t)$ such
that
\[
\eta(t)=0\quad\text{for }t\leq\tau,
\qquad
\eta(t)=1\quad\text{for }t\geq s,
\]
and
\[
|\eta'(t)|
\leq
\frac{C}{s-\tau},
\qquad
|\eta''(t)|
\leq
\frac{C}{(s-\tau)^2},
\]
where $C$ is an absolute constant. Set
$
\wt v:=\eta v,
$
then
\[
\pa_t\wt v
-\omega(x)^{1-p}
\left[
\sum_{i,j=1}^n a^{ij}(x,t)\wt v_{ij}
+\sum_{i=1}^n b^i(x,t)\wt v_i
\right]
+c(x,t)\wt v
=
\omega(x)^q\eta(t) f(x,t)+\eta'(t)v
\]
in $\om\times[0,T]$.

Since $v=0$ on $\pa\om\times[0,T]$, the boundary quotient estimate
implies that
\[
\frac{v}{\omega^q}
\in
\C^\gamma(\om\times[0,T]).
\]
Moreover, $\wt v(\cdot,0)=0$, and hence the compatibility condition at
$t=0$ is automatically satisfied. Therefore,
Theorem~\ref{thm1.1}(ii) gives
\begin{align*}
\|v\|_{\C^{2+\gamma}(\om\times[s,T])}
&\leq
C\left\|
\eta f+\eta'\frac{v}{\omega^q}
\right\|_{\C^\gamma(\om\times[0,T])}\\
&\leq
C(s-\tau)^{-2}
\|Dv\|_{\C^\gamma(\om\times[\tau,T])}
+
C(s-\tau)^{-2}
\|f\|_{\C^\gamma(\om\times[\tau,T])}.
\end{align*}
By an interpolation inequality, there exists $\beta>2$ such that
\[
C(s-\tau)^{-2}
\|Dv\|_{\C^\gamma(\om\times[\tau,T])}
\leq
\frac12
\|v\|_{\C^{2+\gamma}(\om\times[\tau,T])}
+
C(s-\tau)^{-\beta}
\|v\|_{L^r(\om\times[\tau,T])}.
\]
Consequently,
\begin{align*}
\|v\|_{\C^{2+\gamma}(\om\times[s,T])}
&\leq
\frac12
\|v\|_{\C^{2+\gamma}(\om\times[\tau,T])}+
C(s-\tau)^{-\beta}
\left(
\|v\|_{L^r(\om\times[\tau,T])}
+
\|f\|_{\C^\gamma(\om\times[\tau,T])}
\right).
\end{align*}
An application of the iteration lemma
\cite{giaquinta1982regularity}*{Lemma 1.1} yields
\[
\|v\|_{\C^{2+\gamma}(\om\times[s,T])}
\leq
C(s-\tau)^{-\beta}
\left(
\|v\|_{L^r(\om\times[\tau,T])}
+
\|f\|_{\C^\gamma(\om\times[\tau,T])}
\right),
\]
from which the theorem follows.
\end{proof}

\begin{lem}[\cite{jin2024optimal}, Lemma 3.4]\label{lem: ellptical regularity}
Let $p\in(-1,0)$, and suppose that
$
u\in C(\ol\om)\cap C^2(\om)
$
is a solution of
\[
\begin{cases}
-\Delta u(x)=c(x)d(x)^p
&\text{in }\om,\\
u(x)=0
&\text{on }\pa\om,
\end{cases}
\]
where $c\in C^0(\ol\om)$. Then there exists a constant $C>1$,
depending only on $n,p$, and $\om$, such that
\[
\|u\|_{C^{1,p+1}(\ol\om)}
\leq
C\|c\|_{L^\infty(\om)}.
\]
\end{lem}

\begin{proof}
Let $G(\cdot,\cdot)$ denote the Green function of $-\Delta$ in
$\Omega$ with zero Dirichlet boundary condition. Then
\[
u(x)
=
\int_{\Omega}
G(x,y)c(y)d(y)^p\,dy.
\]
The conclusion follows from the standard boundary estimates for
$\nabla_xG$, combined with an elementary dyadic decomposition; see,
for example, \cite{gruterwidman1982}*{Lemma 3.5}. We also refer to
\cite{gilbargtrudinger2001}*{Chapter~4} for the relevant elementary
calculations.
\end{proof}

\begin{thm}\label{thm:regularity of vt in short time}
Let $p\in(-1,0)$, $\gamma\in(0,(1+p)/2)$, and $T>0$. Let $v$ be
the admissible solution of
\[
\pa_tv=v^{1-p}\Delta v
\qquad
\text{in }\om\times[0,T],
\]
satisfying
\begin{equation}\label{eq:upper and lower bound}
\frac{1}{c_0}d(x)
\leq
v(x,t)
\leq
c_0d(x)
\end{equation}
for some constant $c_0>0$. Then
\[
v(x,\cdot)\in C^\infty((0,T])
\qquad
\text{for every }x\in\ol\om,
\]
and
\[
\pa_t^kv(\cdot,t)\in C^{1,p+1}(\ol\om)
\qquad
\text{for all }t\in(0,T]
\text{ and }k\in\mathbb N\cup\{0\}.
\]
Moreover, there exists a constant $C>0$, which depends only on
$n$, $p$, $T$, $k$, $c_0$ and $\|\pa\om\|_{C^{2,(p+1)\gamma/2}}$, such that
\[
\sup_{t\in\left[\frac{T}{2},T\right]}
\|\pa_t^kv(\cdot,t)\|_{C^{1,p+1}(\ol\om)}
\leq C.
\]
\end{thm}

\begin{proof}
By Theorem~\ref{thm:C1alpha of nonlinear equation}, we have
$
v\in
C_{x,t}^{1+\gamma,\frac{1+\gamma}{2}}
\left(
\ol\om\times\left[\frac{T}{20},T\right]
\right).
$
It follows that
\[
\frac{v}{d}
\in
\C^\gamma
\left(
\ol\om\times\left[\frac{T}{20},T\right]
\right).
\]
In particular,
\[
\left\|
\frac{v}{d}
\right\|_{\C^\gamma
(\ol\om\times[\frac{T}{10},T])}
\leq
C
\|v\|_{C_{x,t}^{1+\gamma,\frac{1+\gamma}{2}}
(\om\times[\frac{T}{20},T])}
\leq C.
\]
Viewing the equation as a linear equation with coefficient
$(v/d)^{1-p}$ and applying Theorem~\ref{thm1.1} on $[T/10,T]$,
after translating in time and taking $q=(1-p)/2$, we obtain
\[
\|v\|_{\C^{2+\gamma}
(\ol\om\times[\frac{T}{9},T])}
\leq C,
\]
where $C>0$ depends only on
$n,p,c_0,\gamma,T$, and
$\|\pa\om\|_{C^{2,(p+1)\gamma/2}}$.

We next prove by a bootstrap argument that there exists a constant
$C>0$, depending only on
$n,p,c_0,\gamma,T$, and
$\|\pa\om\|_{C^{2,(p+1)\gamma/2}}$, such that
\[
\|\pa_tv\|_{\C^{2+\gamma}
(\ol\om\times[\frac{T}{5},T])}
\leq C.
\]
Let
\[
v_\lambda^h(x,t)
:=
\frac{v(x,t)-v(x,t-h)}{h^\lambda},
\]
where $\lambda\in(0,1]$ and
$h\in(0,T/100)$ is sufficiently small. The equation for $v$ gives
\begin{equation}\label{eq:equation of v lambda h}
\pa_tv_\lambda^h
-
\left(\frac{v}{d}\right)^{1-p}
d^{1-p}\Delta v_\lambda^h
=
d^{\frac{1-p}{2}}f
\qquad
\text{in }\om\times\left(\frac{T}{100},T\right],
\end{equation}
where
\begin{align*}
f(x,t)
&=
d(x)^{\frac{p-1}{2}}
\frac{
v(x,t)^{1-p}-v(x,t-h)^{1-p}
}{h^\lambda}
\Delta v(x,t-h)\\
&=
d(x)^{\frac{p-1}{2}}
\Delta v(x,t-h)
\frac{1}{h^\lambda}
\int_0^1
\frac{d}{ds}
\left[
sv(x,t)+(1-s)v(x,t-h)
\right]^{1-p}
\,ds\\
&=
d(x)^{\frac{p-1}{2}}
\Delta v(x,t-h)
\frac{v(x,t)-v(x,t-h)}{h^\lambda}\\
&\qquad\cdot
\int_0^1
(1-p)
\left[
sv(x,t)+(1-s)v(x,t-h)
\right]^{-p}
\,ds\\
&=
\left(
d(x)^{\frac{1-p}{2}}
\Delta v(x,t-h)
\right)
\frac{v_\lambda^h(x,t)}{d(x)}\\
&\qquad\cdot
\int_0^1
(1-p)
\left[
s\frac{v(x,t)}{d(x)}
+
(1-s)\frac{v(x,t-h)}{d(x)}
\right]^{-p}
\,ds.
\end{align*}

Let $\lambda_0=\gamma/2$, and choose a finite sequence
\[
\lambda_0<\lambda_1<\cdots<\lambda_N=1
\]
such that
\[
\lambda_{j+1}-\lambda_j<\frac{\gamma}{2},
\qquad
j=0,\ldots,N-1.
\]
By the regularity of $v$, we have
\[
\left|
\frac{v_{\lambda_0}^h}{d}
\right|
\leq C
\qquad
\text{in }\om\times\left[\frac{T}{8},T\right].
\]
By a standard localization and boundary-flattening argument, together
with the usual interior estimates and a finite covering,
Theorem~\ref{thm:C1alpha of nonlinear equation in Q1}, applied to
\eqref{eq:equation of v lambda h}, yields
\[
\|v_{\lambda_0}^h\|_
{C_{x,t}^{1+\gamma,\frac{1+\gamma}{2}}
(\ol\om\times[\frac{T}{8},T])}
\leq C.
\]
Since $v_{\lambda_0}^h=0$ on
$\pa\om\times[T/8,T]$, it follows that
\[
\left\|
\frac{v_{\lambda_0}^h}{d}
\right\|_
{\C^\gamma(\ol\om\times[\frac{T}{8},T])}
\leq C.
\]
Moreover, since
$
v\in
\C^{2+\gamma}
\left(
\ol\om\times\left[T/9,T\right]
\right),
$
we have
\[
\frac{v}{d},
\quad
d(x)^{\frac{1-p}{2}}\Delta v(x,t-h)
\in
\C^\gamma
\left(
\ol\om\times\left[\frac{T}{8},T\right]
\right).
\]
Hence,
\[
\|f\|_{\C^\gamma
(\ol\om\times[\frac{T}{8},T])}
\leq C.
\]
Consequently, Theorem~\ref{thm1.1} gives
\[
\|v_{\lambda_0}^h\|_{\C^{2+\gamma}
(\ol\om\times\left[\frac{T}{7},T\right])}
\leq C
\]
uniformly in $h$. By the standard finite-difference improvement
result and the choice
$
\lambda_1-\lambda_0<\frac{\gamma}{2},
$
we obtain, on a slightly smaller time interval,
\[
\left|
\frac{v_{\lambda_1}^h}{d}
\right|
\leq C.
\]
Repeating the preceding argument finitely many times on nested time
intervals, we obtain
\[
\sup_{0<h<\frac{T}{100}}
\|v_1^h\|_{\C^{2+\gamma}
(\ol\om\times\left[\frac{T}{5},T\right])}
\leq C.
\]
Letting $h\to0$, we conclude that
\[
\|\pa_tv\|_{\C^{2+\gamma}
(\ol\om\times\left[\frac{T}{5},T\right])}
\leq
C\left(
n,p,c_0,\gamma,T, \|\pa\om\|_{C^{2,(p+1)\gamma/2}}
\right).
\]

Let 
$
w_k:=\pa_t^kv,
$ 
and 
$
q_k:=w_k/v.
$
Repeatedly differentiating the equation for $v$ with respect to time,
we obtain
\[
\pa_tw_k
-v^{1-p}\Delta w_k
+c_k(q_1)w_k
=
d(x)F_k
\left(
\frac{v}{d},q_1,\ldots,q_{k-1}
\right),
\]
where $c_k$ and $F_k$ are smooth functions depending only on $p$ and
$k$. Fix $l\in\mathbb N$. Applying Theorem~\ref{thm1.1} inductively
on finitely many nested time intervals gives
\[
\max_{1\leq k\leq l+1}
\|w_k\|_{\C^{2+\gamma}
(\ol\om\times\left[\frac{T}{2},T\right])}
\leq
C\left(
n,p,c_0,\gamma,T,
\|\pa\om\|_{C^{2,(p+1)\gamma/2}},l
\right).
\]

On the other hand,
\[
-\Delta w_l
=
-d(x)^p
\left(
\frac{v}{d}
\right)^{p-1}\left[F_l
\left(
\frac{v}{d},q_1,\ldots,q_{l-1}
\right)-c_l(q_1) \frac{w_l}{d}- \frac{w_{l+1}}{d}\right].
\]
Therefore,
Lemma~\ref{lem: ellptical regularity} yields
\[
\sup_{t\in\left[\frac{T}{2},T\right]}
\|\pa_t^lv(\cdot,t)\|_{C^{1,p+1}(\ol\om)}
\leq
C\left(
n,p,c_0,\gamma,T,
\|\pa\om\|_{C^{2,(p+1)\gamma/2}},l
\right).
\]
\end{proof}

\begin{thm}\label{thm: v controlled by S}
Let $p\in(-1,0)$, $\gamma\in(0,(1+p)/2)$, and
$v_0\in C(\ol{\om})$ satisfy \eqref{eq:inequality for v0} for some
constant $c_0$. Then \eqref{eq:equation for v} admits a unique
admissible solution
$
v\in\C^{2+\gamma}(\ol{\om}\times(0,+\infty)),
$
and
\be\label{eq: long time existence of v}
\frac{1}{C_0}(1+t)^{\frac{1}{p-1}}S(x)
\leq
v(x,t)
\leq
C_0(1+t)^{\frac{1}{p-1}}S(x)
\quad
\text{on }\ol{\om}\times[0,+\infty),
\ee
for some constant $C_0>0$ depending only on
$n,p,c_0$, and $\|\pa\om\|_{C^{2,(p+1)\gamma/2}}$, where $S$ is the
unique solution of \eqref{eq:equation of S}.

Moreover, for every $x\in\ol{\om}$,
\[
v(x,\cdot)\in C^\infty((0,+\infty)),
\]
and, for every $t\in(0,+\infty)$ and
$k\in\mathbb N\cup\{0\}$,
\[
\pa_t^kv(\cdot,t)\in C^{1,p+1}(\ol{\om}).
\]
Finally, for every $\delta>0$ and
$k\in\mathbb N\cup\{0\}$, there exists a constant $C>0$, depending
only on
$n,C_0,p,k,\delta$, and $\|\pa\om\|_{C^{2,(p+1)\gamma/2}}$, such that
\be\label{eq:decay estimates of higher order regularity}
\left\|
\pa_t^kv(\cdot,t)
\right\|_{C^{1,p+1}(\ol{\om})}
\leq
C(1+t)^{\frac{1}{p-1}-k}
\quad
\text{for all }t\in[\delta,+\infty).
\ee
\end{thm}

\begin{proof}
Since \(v_0\) and \(S\) are both comparable to \(d\), there exist
\(t_1,t_2>0\), depending only on
\(n,p,c_0\) and \(\|\pa\om\|_{C^{2,(p+1)\gamma/2}}\), such that
\[
t_1^{\frac{1}{p-1}}S(x)
\leq v_0(x)
\leq t_2^{\frac{1}{p-1}}S(x)
\quad \text{on } \ol{\om}.
\]
For every \(T>0\), Theorem \ref{thm:existence} gives an admissible
solution \(v^T\) on \(\ol{\om}\times[0,T]\). Comparing \(v^T\) with the
solutions
\[
(t+t_i)^{\frac{1}{p-1}}S(x),\qquad i=1,2,
\]
and applying Theorem \ref{thm:comparison}, we obtain
\[
(t+t_1)^{\frac{1}{p-1}}S(x)
\leq v^T(x,t)
\leq (t+t_2)^{\frac{1}{p-1}}S(x)
\quad \text{on } \ol{\om}\times[0,T].
\]
By uniqueness, these finite-time solutions agree on their common time
intervals and therefore define a unique admissible solution \(v\) on
\(\ol{\om}\times[0,+\infty)\). In particular,
\be\label{eq:barrier estimates}
(t+t_1)^{\frac{1}{p-1}}S(x)
\leq v(x,t)
\leq (t+t_2)^{\frac{1}{p-1}}S(x)
\quad \text{on } \ol{\om}\times[0,+\infty).
\ee
Since \(t+t_i\) is comparable to \(1+t\), this proves
\eqref{eq: long time existence of v}. The stated regularity follows from
Theorem \ref{thm:regularity of vt in short time}, applied on arbitrary
finite time intervals.

It remains to prove
\eqref{eq:decay estimates of higher order regularity}. For any
\(t_0\geq\delta\), define
\[
\wt{v}(x,t)
=
t_0^{\frac{1}{1-p}}v(x,t_0t),
\qquad
(x,t)\in\ol{\om}\times\left[\frac12,1\right].
\]
Then \(\wt{v}\) satisfies the same equation. By
\eqref{eq: long time existence of v} and the comparability of \(S\) and
\(d\), there exists \(c_\delta>1\), independent of \(t_0\), such that
\[
\frac{1}{c_\delta}d(x)
\leq \wt{v}(x,t)
\leq c_\delta d(x)
\quad\text{on }
\ol{\om}\times\left[\frac12,1\right].
\]
Therefore, after a time translation, Theorem
\ref{thm:regularity of vt in short time} gives
\[
\left\|\pa_t^k\wt{v}(\cdot,1)\right\|_{C^{1,p+1}(\ol{\om})}
\leq C,
\]
where \(C\) is independent of \(t_0\). Since
\[
\pa_t^k\wt{v}(x,1)
=
t_0^{\frac{1}{1-p}+k}\pa_t^k v(x,t_0),
\]
we obtain
\[
\left\|\pa_t^k v(\cdot,t_0)\right\|_{C^{1,p+1}(\ol{\om})}
\leq
Ct_0^{\frac{1}{p-1}-k}
\leq
C(1+t_0)^{\frac{1}{p-1}-k}.
\]
Since \(t_0\geq\delta\) is arbitrary, this proves
\eqref{eq:decay estimates of higher order regularity}.
\end{proof}

\begin{proof}[Proof of Theorem \ref{thm:main thm 1}.]
The existence, uniqueness, and all the stated regularity properties
follow from Theorem~\ref{thm: v controlled by S}.

To see the optimality, let $\om=(0,1)$ and consider the friendly giant
solution
\[
V(x,t)
=
(1+t)^{\frac{1}{p-1}}S(x).
\]
Since $S(x)\asymp x$ near $x=0$ and
\[
-S''
=
\frac{1}{1-p}S^p,
\]
we have
\[
S'(0)-S'(x)
=
\frac{1}{1-p}
\int_0^x S(s)^p\,ds
\asymp
x^{p+1}.
\]
Consequently,
\[
V(\cdot,t)
\notin
C^{1,p+1+\varepsilon}([0,1])
\]
for every $\varepsilon\in(0,-p)$ and every $t>0$. Thus, this example
already demonstrates the optimality of the exponent $p+1$ in
\eqref{eq:main thm 1 regularity in x}.
\end{proof}

\begin{rem}\label{cor: regularity of v over Sm}
Let $\om$ be a bounded smooth domain, let $p\in(-1,0)$, and let
$v_0\in C(\ol\om)$ be positive and satisfy
\eqref{eq:inequality for v0} for some constant $c_0$. If $v$ is a
nonnegative admissible solution of \eqref{eq:equation for v}, then
\be\label{eq:regularity for v over Sm}
\frac{v}{S}\in C^{p+1}(\ol{\om})
\qquad
\text{for all }t>0,
\ee
where $S$ is given by \eqref{eq:equation of S}. This follows from the
same argument as in \cite{jin2024optimal}*{Corollary 1.2}.

Moreover, the exponent $p+1$ in
\eqref{eq:regularity for v over Sm} is optimal, as can be seen by
adapting the construction in
\cite{jin2024optimal}*{Example 4.2}.
\end{rem}

Finally, we investigate the large-time asymptotics. Since the details are essentially the same as those in \cite{jin2024optimal}, we present only the main steps.

\begin{proof}[Proof of Theorem \ref{main thm 1.3}.]
Let $v$ be a solution of \eqref{eq:equation for v}, then
we rescale variables as follows:
\be\label{eq:rescale giant solution}
\theta(x,\tau):=t^{\frac{1}{1-p}}v(x,t)\quad \text{with}~t=e^\tau.
\ee
Then
\[
\pa_\tau \theta^p =p\Delta \theta+\frac{p}{1-p}\theta^p\quad \text{in}~\om\times (0,\infty).
\]
It follows from \eqref{eq:barrier estimates} that
\be\label{eq:L infty convergence of theta}
\left\| \frac{\theta(\cdot,\tau)}{S(\cdot)}-1 \right\|_{L^\infty(\om)}\leq Ct^{-1} \leq Ce^{-\tau}\quad \text{for all } \tau>1,
\ee
where  $S$ is from \eqref{eq:equation of S}. 
Moreover, for all $\tau\geq 2$,
\be\label{eq:C p+1 convergence of  theta and Theta}
\begin{split}
\left\|\frac{\theta(\cdot,\tau)-S}{S}\right\|_{C^{p+1}(\ol{\om})}
& \leq C e^{-\tau}.
\end{split}
\ee

The estimate \eqref{eq:decay rate of v} follows immediately from
Theorem~\ref{thm: v controlled by S}. For the remaining assertions, we
follow the spectral-decomposition argument in
\cite{jin2024optimal}*{Section 5}. We only indicate the
ingredients that differ in the present setting.

Notice that \(h=\theta-S\) satisfies
\[
S^{p-1}h_\tau=-\mathscr L_S h+N(h),
\]
where
\[
N(h)
=
\frac{1}{1-p}S^p
\left[
\left(1+\frac{h}{S}\right)^p-1-p\frac{h}{S}
\right]
+
S^{p-1}
\left[
1-\left(1+\frac{h}{S}\right)^{p-1}
\right]h_\tau .
\]
The estimates obtained above imply, for all sufficiently large \(\tau\),
\[
|N(h)|\leq CS^p e^{-2\tau},
\qquad
\|N(h)\|_{\C^{p+1}(\Omega\times[\tau-1,\tau])}
\leq Ce^{-2\tau}.
\]

By Hardy's inequality, the embedding
\[
H_0^1(\Omega)
\hookrightarrow
L^2(\Omega;S^{p-1}\,dx)
\]
is compact. Hence, \(\mathscr L_S\) admits the required spectral
decomposition, with first eigenpair
\[
\mu_1=1,
\qquad
\psi_1
=
\frac{S}{\|S\|_{L^2(\Omega;S^{p-1}\,dx)}}.
\]
The spectral projection argument in
\cite{jin2024optimal}*{proof of Theorem~1.3} then yields constants
\(A_1\geq0\) and \(\gamma_1>0\) such that, upon setting
\[
\widetilde h(x,\tau)
:=
\theta(x,\tau)-S(x)+A_1e^{-\tau}S(x),
\]
we have
\[
\|\widetilde h(\cdot,\tau)\|_{L^2(\Omega;S^{p-1}\,dx)}
\leq Ce^{-(1+\gamma_1)\tau}.
\]
Since \(S\) is the first eigenfunction, \(\widetilde h\) satisfies
\begin{equation}\label{eq: the equation of widetilde h}
\begin{cases}
S^{p-1}\widetilde h_\tau
=
-\mathscr L_S\widetilde h+N(h)
&\text{in }\Omega\times(1,+\infty),\\
\widetilde h=0
&\text{on }\partial\Omega\times(1,+\infty).
\end{cases}
\end{equation}
Applying Theorem~\ref{thm:cut off estimate} to
\eqref{eq: the equation of widetilde h} gives
\[
\|\widetilde h(\cdot,\tau)\|_{C^{1,p+1}(\overline\Omega)}
\leq Ce^{-(1+\gamma_1)\tau}.
\]
Changing back to \(t=e^\tau\) proves \eqref{eq:expansion of v in long time}, then  \eqref{eq:long time convergence of u} follows.
\end{proof}

\begin{rem}
In the same way as  in \cite{jin2024optimal}*{Remark 5.1}, we can also expand the solution up to any order.
\end{rem}

\small


\begin{bibdiv}
\begin{biblist}

\bib{akagi2016stability}{article}{
   author={Akagi, Goro},
   title={Stability of non-isolated asymptotic profiles for fast diffusion},
   journal={Comm. Math. Phys.},
   volume={345},
   date={2016},
   number={1},
   pages={77--100},
   issn={0010-3616},
   review={\MR{3509010}},
}

\bib{akagi2023rates}{article}{
   author={Akagi, Goro},
   title={Rates of convergence to non-degenerate asymptotic profiles for
   fast diffusion via energy methods},
   journal={Arch. Ration. Mech. Anal.},
   volume={247},
   date={2023},
   number={2},
   pages={Paper No. 23, 38},
   issn={0003-9527},
   review={\MR{4565026}},
}

\bib{aronson1979regularite}{article}{
   author={Aronson, Donald G.},
   author={B\'enilan, Philippe},
   title={R\'egularit\'e{} des solutions de l'\'equation des milieux poreux
   dans ${\bf R}\sp{N}$},
   language={French, with English summary},
   journal={C. R. Acad. Sci. Paris S\'er. A-B},
   volume={288},
   date={1979},
   number={2},
   pages={A103--A105},
   issn={0151-0509},
   review={\MR{0524760}},
}

\bib{aronson1981large}{article}{
   author={Aronson, D.~G.},
   author={Peletier, L.~A.},
   title={Large time behaviour of solutions of the porous medium equation
   in bounded domains},
   date={1981},
   ISSN={0022-0396,1090-2732},
   journal={J. Differential Equations},
   volume={39},
   number={3},
   pages={378\ndash 412},
   url={https://doi.org/10.1016/0022-0396(81)90065-6},
   review={\MR{612594}},
}



\bib{bekmaganbetov2025singulardegenerate}{article}{
   author={Bekmaganbetov, Bekarys},
   author={Dong, Hongjie},
   title={Singular-degenerate parabolic systems with the conormal boundary
   condition on the upper half space},
   date={2025},
   journal={arXiv:2509.18418},
   url={https://arxiv.org/abs/2509.18418},
}

\bib{benilan1981regularizing}{article}{
   author={B\'enilan, Philippe},
   author={Crandall, Michael G.},
   title={Regularizing effects of homogeneous evolution equations},
   conference={
      title={Contributions to analysis and geometry},
      address={Baltimore, Md.},
      date={1980},
   },
   book={publisher={Johns Hopkins Univ. Press, Baltimore, MD},},
   isbn={0-8018-2779-5},
   date={1981},
   pages={23--39},
   review={\MR{0648452}},
}

\bib{berryman1980stability}{article}{
   author={Berryman, James~G.},
   author={Holland, Charles~J.},
   title={Stability of the separable solution for fast diffusion},
   date={1980},
   ISSN={0003-9527},
   journal={Arch. Rational Mech. Anal.},
   volume={74},
   number={4},
   pages={379\ndash 388},
   url={https://doi.org/10.1007/BF00249681},
   review={\MR{588035}},
}

\bib{bertsch1992nonuniqueness}{article}{
   author={Bertsch, Michiel},
   author={Dal Passo, Roberta},
   author={Ughi, Maura},
   title={Nonuniqueness of solutions of a degenerate parabolic equation},
   journal={Ann. Mat. Pura Appl. (4)},
   volume={161},
   date={1992},
   pages={57--81},
   issn={0003-4622},
   review={\MR{1174811}},
}

\bib{bonforte2006global}{article}{
   author={Bonforte, Matteo},
   author={Vazquez, Juan Luis},
   title={Global positivity estimates and Harnack inequalities for the fast
   diffusion equation},
   journal={J. Funct. Anal.},
   volume={240},
   date={2006},
   number={2},
   pages={399--428},
   issn={0022-1236},
   review={\MR{2261689}},
}

\bib{bonforte2010positivity}{article}{
   author={Bonforte, Matteo},
   author={V\'azquez, Juan Luis},
   title={Positivity, local smoothing, and Harnack inequalities for very
   fast diffusion equations},
   journal={Adv. Math.},
   volume={223},
   date={2010},
   number={2},
   pages={529--578},
   issn={0001-8708},
   review={\MR{2565541}},
}

\bib{bonforte2012behaviour}{article}{
   author={Bonforte, Matteo},
   author={Grillo, Gabriele},
   author={Vazquez, Juan Luis},
   title={Behaviour near extinction for the Fast Diffusion Equation on
   bounded domains},
   language={English, with English and French summaries},
   journal={J. Math. Pures Appl. (9)},
   volume={97},
   date={2012},
   number={1},
   pages={1--38},
   issn={0021-7824},
   review={\MR{2863762}},
}

\bib{bonforte2021sharp}{article}{
   author={Bonforte, Matteo},
   author={Figalli, Alessio},
   title={Sharp extinction rates for fast diffusion equations on generic
   bounded domains},
   journal={Comm. Pure Appl. Math.},
   volume={74},
   date={2021},
   number={4},
   pages={744--789},
   issn={0010-3640},
   review={\MR{4221933}},
}

\bib{bonforte2024cauchy}{article}{
   author={Bonforte, Matteo},
   author={Figalli, Alessio},
   title={The Cauchy-Dirichlet problem for the fast diffusion equation on
   bounded domains},
   journal={Nonlinear Anal.},
   volume={239},
   date={2024},
   pages={Paper No. 113394, 55},
   issn={0362-546X},
   review={\MR{4658530}},
}

\bib{brezis1983nonlinear}{article}{
   author={Br\'ezis, Ha\"im},
   author={Friedman, Avner},
   title={Nonlinear parabolic equations involving measures as initial
   conditions},
   journal={J. Math. Pures Appl. (9)},
   volume={62},
   date={1983},
   number={1},
   pages={73--97},
   issn={0021-7824},
   review={\MR{0700049}},
}

\bib{caffarelli1980regularity}{article}{
   author={Caffarelli, Luis A.},
   author={Friedman, Avner},
   title={Regularity of the free boundary of a gas flow in an
   $n$-dimensional porous medium},
   journal={Indiana Univ. Math. J.},
   volume={29},
   date={1980},
   number={3},
   pages={361--391},
   issn={0022-2518},
   review={\MR{0570687}},
}

\bib{caffarelli1987lipschitz}{article}{
   author={Caffarelli, L. A.},
   author={V\'azquez, J. L.},
   author={Wolanski, N. I.},
   title={Lipschitz continuity of solutions and interfaces of the
   $N$-dimensional porous medium equation},
   journal={Indiana Univ. Math. J.},
   volume={36},
   date={1987},
   number={2},
   pages={373--401},
   issn={0022-2518},
   review={\MR{0891781}},
}

\bib{caffarelli1990regularity}{article}{
   author={Caffarelli, Luis A.},
   author={Wolanski, Noem\'i\ I.},
   title={$C^{1,\alpha}$ regularity of the free boundary for the
   $N$-dimensional porous media equation},
   journal={Comm. Pure Appl. Math.},
   volume={43},
   date={1990},
   number={7},
   pages={885--902},
   issn={0010-3640},
   review={\MR{1072396}},
}

\bib{chasseigne2002theory}{article}{
   author={Chasseigne, Emmanuel},
   author={Vazquez, Juan Luis},
   title={Theory of extended solutions for fast-diffusion equations in
   optimal classes of data. Radiation from singularities},
   journal={Arch. Ration. Mech. Anal.},
   volume={164},
   date={2002},
   number={2},
   pages={133--187},
   issn={0003-9527},
   review={\MR{1929929}},
   doi={10.1007/s00205-002-0210-0},
}


\bib{chasseigne2003pressure}{article}{
   author={Chasseigne, Emmanuel},
   author={V\'azquez, Juan Luis},
   title={The pressure equation in the fast diffusion range},
   journal={Rev. Mat. Iberoamericana},
   volume={19},
   date={2003},
   number={3},
   pages={873--917},
   issn={0213-2230},
   review={\MR{2053567}},
}

\bib{choi2023asymptotics}{article}{
   author={Choi, Beomjun},
   author={McCann, Robert J.},
   author={Seis, Christian},
   title={Asymptotics near extinction for nonlinear fast diffusion on a
   bounded domain},
   journal={Arch. Ration. Mech. Anal.},
   volume={247},
   date={2023},
   number={2},
   pages={Paper No. 16, 48},
   issn={0003-9527},
   review={\MR{4553941}},
}



\bib{choi2024finitedimensional}{article}{
   author={Choi, Beomjun},
   author={Seis, Christian},
   title={Finite-dimensional leading order dynamics for the fast diffusion
   equation near extinction},
   journal={Discrete Contin. Dyn. Syst.},
   volume={44},
   date={2024},
   number={9},
   pages={2697--2712},
   issn={1078-0947},
   review={\MR{4762615}},
}

\bib{dahlberg1988nonnegative}{article}{
   author={Dahlberg, Bj\"orn E. J.},
   author={Kenig, Carlos E.},
   title={Nonnegative solutions of the initial-Dirichlet problem for
   generalized porous medium equations in cylinders},
   journal={J. Amer. Math. Soc.},
   volume={1},
   date={1988},
   number={2},
   pages={401--412},
   issn={0894-0347},
   review={\MR{0928264}},
   doi={10.2307/1990922},
}

\bib{daskalopoulos1997nonlinear}{article}{
   author={Daskalopoulos, Panagiota},
   author={Del Pino, Manuel},
   title={On nonlinear parabolic equations of very fast diffusion},
   journal={Arch. Rational Mech. Anal.},
   volume={137},
   date={1997},
   number={4},
   pages={363--380},
   issn={0003-9527},
   review={\MR{1463800}},
}

\bib{daskalopoulos1998regularity}{article}{
   author={Daskalopoulos, P.},
   author={Hamilton, R.},
   title={Regularity of the free boundary for the porous medium equation},
   journal={J. Amer. Math. Soc.},
   volume={11},
   date={1998},
   number={4},
   pages={899--965},
   issn={0894-0347},
   review={\MR{1623198}},
}

\bib{daskalopoulos2001all}{article}{
   author={Daskalopoulos, P.},
   author={Hamilton, R.},
   author={Lee, K.},
   title={All time $C^\infty$-regularity of the interface in degenerate
   diffusion: a geometric approach},
   journal={Duke Math. J.},
   volume={108},
   date={2001},
   number={2},
   pages={295--327},
   issn={0012-7094},
   review={\MR{1833393}},
}

\bib{daskalopoulos2007degenerate}{book}{
   author={Daskalopoulos, Panagiota},
   author={Kenig, Carlos E.},
   title={Degenerate diffusions},
   series={EMS Tracts in Mathematics},
   volume={1},
   note={Initial value problems and local regularity theory},
   publisher={European Mathematical Society (EMS), Z\"urich},
   date={2007},
   pages={x+198},
   isbn={978-3-03719-033-3},
   review={\MR{2338118}},
   doi={10.4171/033},
}

\bib{dibenedetto1991local}{article}{
   author={DiBenedetto, E.},
   author={Kwong, Y.},
   author={Vespri, V.},
   title={Local space-analyticity of solutions of certain singular
   parabolic equations},
   date={1991},
   ISSN={0022-2518,1943-5258},
   journal={Indiana Univ. Math. J.},
   volume={40},
   number={2},
   pages={741\ndash 765},
   url={https://doi.org/10.1512/iumj.1991.40.40033},
   review={\MR{1119195}},
}

\bib{dibenedetto1993degenerate}{book}{
   author={DiBenedetto, Emmanuele},
   title={Degenerate parabolic equations},
   series={Universitext},
   publisher={Springer-Verlag, New York},
   date={1993},
   pages={xvi+387},
   isbn={0-387-94020-0},
   review={\MR{1230384}},
}

\bib{dong2021parabolic}{article}{
   author={Dong, Hongjie},
   author={Phan, Tuoc},
   title={Parabolic and elliptic equations with singular or degenerate
   coefficients: the {D}irichlet problem},
   date={2021},
   ISSN={0002-9947,1088-6850},
   journal={Trans. Amer. Math. Soc.},
   volume={374},
   number={9},
   pages={6611\ndash 6647},
   url={https://doi.org/10.1090/tran/8397},
   review={\MR{4302171}},
}

\bib{dong2023degenerate}{article}{
   author={Dong, Hongjie},
   author={Phan, Tuoc},
   author={Tran, Hung~Vinh},
   title={Degenerate linear parabolic equations in divergence form on the
   upper half space},
   date={2023},
   ISSN={0002-9947,1088-6850},
   journal={Trans. Amer. Math. Soc.},
   volume={376},
   number={6},
   pages={4421\ndash 4451},
   url={https://doi.org/10.1090/tran/8892},
   review={\MR{4586816}},
}

\bib{dong2023parabolic}{article}{
   author={Dong, Hongjie},
   author={Phan, Tuoc},
   title={On parabolic and elliptic equations with singular or degenerate
   coefficients},
   date={2023},
   ISSN={0022-2518,1943-5258},
   journal={Indiana Univ. Math. J.},
   volume={72},
   number={4},
   pages={1461\ndash 1502},
   review={\MR{4637368}},
}

\bib{dong2024nondivergence}{article}{
   author={Dong, Hongjie},
   author={Phan, Tuoc},
   author={Tran, Hung~Vinh},
   title={Nondivergence form degenerate linear parabolic equations on the
   upper half space},
   date={2024},
   ISSN={0022-1236,1096-0783},
   journal={J. Funct. Anal.},
   volume={286},
   number={9},
   pages={Paper No. 110374, 53pp},
   url={https://doi.org/10.1016/j.jfa.2024.110374},
   review={\MR{4708668}},
}

\bib{dong2025schaudertypeestimatesdegenerate}{article}{
   author={Dong, Hongjie},
   author={Jeon, Seongmin},
   title={Schauder type estimates for degenerate or singular parabolic
   systems with partially dmo coefficients},
   journal={Rev. Mat. Iberoam. (2025), published online first},
   url={https://arxiv.org/abs/2502.08926},
}

\bib{dong2025nondivergence}{article}{
   author={Dong, Hongjie},
   author={Ryu, Junhee},
   title={On nondivergence form linear parabolic and elliptic equations with
   degenerate coefficients},
   journal={J. Funct. Anal.},
   volume={291},
   date={2026},
   number={3},
   pages={Paper No. 111511, 39pp},
   issn={0022-1236},
   review={\MR{5062413}},
}

\bib{dong2026degenerate}{article}{
   author={Dong, Hongjie},
   author={Jeon, Seongmin},
   title={Degenerate or singular parabolic systems with partially {DMO}
   coefficients: the {D}irichlet problem},
   date={2026},
   ISSN={0022-1236,1096-0783},
   journal={J. Funct. Anal.},
   volume={290},
   number={9},
   pages={Paper No. 111365, 48pp},
   url={https://doi.org/10.1016/j.jfa.2026.111365},
   review={\MR{5023999}},
}

\bib{dong2026sobolev}{article}{
   author={Dong, Hongjie},
   author={Ryu, Junhee},
   title={Sobolev estimates for parabolic and elliptic equations in
   divergence form with degenerate coefficients},
   date={2026},
   ISSN={0002-9947,1088-6850},
   journal={Trans. Amer. Math. Soc.},
   volume={379},
   number={2},
   pages={1283\ndash 1326},
   url={https://doi.org/10.1090/tran/9527},
   review={\MR{5019582}},
}

\bib{feireisl2000convergence}{article}{
   author={Feireisl, Eduard},
   author={Simondon, Fr\'ed\'erique},
   title={Convergence for semilinear degenerate parabolic equations in
   several space dimensions},
   journal={J. Dynam. Differential Equations},
   volume={12},
   date={2000},
   number={3},
   pages={647--673},
   issn={1040-7294},
   review={\MR{1800136}},
}

\bib{galaktionov2002fast}{article}{
   author={Galaktionov, Victor A.},
   author={King, John R.},
   title={Fast diffusion equation with critical Sobolev exponent in a ball},
   journal={Nonlinearity},
   volume={15},
   date={2002},
   number={1},
   pages={173--188},
   issn={0951-7715},
   review={\MR{1877973}},
}

\bib{giaquinta1982regularity}{article}{
   author={Giaquinta, Mariano},
   author={Giusti, Enrico},
   title={On the regularity of the minima of variational integrals},
   journal={Acta Math.},
   volume={148},
   date={1982},
   pages={31--46},
   issn={0001-5962},
   review={\MR{0666107}},
}

\bib{gilbargtrudinger2001}{book}{
    author={Gilbarg, David},
    author={Trudinger, Neil S.},
    title={Elliptic partial differential equations of second order},
    edition={2},
    series={Classics in Mathematics},
    publisher={Springer-Verlag},
    place={Berlin},
    date={2001},
    note={Reprint of the 1998 edition},
    doi={10.1007/978-3-642-61798-0},
}

\bib{gruterwidman1982}{article}{
    author={Gr{\"u}ter, Michael},
    author={Widman, Kjell-Ove},
    title={The Green function for uniformly elliptic equations},
    journal={Manuscripta Math.},
    volume={37},
    date={1982},
    number={3},
    pages={303--342},
    doi={10.1007/BF01166225},
}

\bib{gui1993regularity}{article}{
   author={Gui, Changfeng},
   author={Lin, Fang-Hua},
   title={Regularity of an elliptic problem with a singular nonlinearity},
   journal={Proc. Roy. Soc. Edinburgh Sect. A},
   volume={123},
   date={1993},
   number={6},
   pages={1021--1029},
   issn={0308-2105},
   review={\MR{1263903}},
}

\bib{herrero1985cauchy}{article}{
   author={Herrero, Miguel A.},
   author={Pierre, Michel},
   title={The Cauchy problem for $u_t=\Delta u^m$ when $0<m<1$},
   journal={Trans. Amer. Math. Soc.},
   volume={291},
   date={1985},
   number={1},
   pages={145--158},
   issn={0002-9947},
   review={\MR{0797051}},
   doi={10.2307/1999900},
}

\bib{hissinkmuller2022wellposedness}{article}{
   author={Hissink Muller, V.},
   author={Sonner, S.},
   title={Well-posedness of singular-degenerate porous medium type
   equations and application to biofilm models},
   date={2022},
   ISSN={0022-247X,1096-0813},
   journal={J. Math. Anal. Appl.},
   volume={509},
   number={1},
   pages={Paper No. 125894, 34 pp.},
   url={https://doi.org/10.1016/j.jmaa.2021.125894},
   review={\MR{4354866}},
}

\bib{jin2023bubbling}{article}{
   author={Jin, Tianling},
   author={Xiong, Jingang},
   title={Bubbling and extinction for some fast diffusion equations in
   bounded domains},
   journal={Trans. Amer. Math. Soc. Ser. B},
   volume={10},
   date={2023},
   pages={1287--1332},
   review={\MR{4642416}},
}

\bib{jin2023optimal}{article}{
   author={Jin, Tianling},
   author={Xiong, Jingang},
   title={Optimal boundary regularity for fast diffusion equations in
   bounded domains},
   journal={Amer. J. Math.},
   volume={145},
   date={2023},
   number={1},
   pages={151--219},
   issn={0002-9327},
   review={\MR{4545845}},
}

\bib{jin2024optimal}{article}{
   author={Jin, Tianling},
   author={Ros-Oton, Xavier},
   author={Xiong, Jingang},
   title={Optimal regularity and fine asymptotics for the porous medium
   equation in bounded domains},
   journal={J. Reine Angew. Math.},
   volume={809},
   date={2024},
   pages={269--300},
   issn={0075-4102},
   review={\MR{4726572}},
}

\bib{jin2025regularity}{article}{
   author={Jin, Tianling},
   author={Xiong, Jingang},
   title={Regularity of solutions to the Dirichlet problem for fast
   diffusion equations},
   journal={Adv. Math.},
   volume={478},
   date={2025},
   pages={Paper No. 110390, 32},
   issn={0001-8708},
   review={\MR{4915011}},
}

\bib{jin2026bubbledynamics}{article}{
   author={Jin, Tianling},
   author={Xiong, Jingang},
   title={Extinction profiles for the Sobolev critical fast diffusion
   equation in bounded domains. I. One bubble dynamics},
   journal={Arch. Ration. Mech. Anal.},
   volume={250},
   date={2026},
   number={2},
   pages={Paper No. 14, 43},
   issn={0003-9527},
   review={\MR{5043140}},
}

\bib{JTXZ}{article}{
   author={Jin, Tianling},
   author={Tu, Xushan},
   author={Xiong, Jingang},
   author={Zheng, Zhen},
   title={Schauder estimates for the linearized very fast diffusion
   equations in bounded domains},
   journal={preprint},
   date={2026},
}

\bib{kienzler2018flatness}{article}{
   author={Kienzler, Clemens},
   author={Koch, Herbert},
   author={V\'azquez, Juan Luis},
   title={Flatness implies smoothness for solutions of the porous medium
   equation},
   journal={Calc. Var. Partial Differential Equations},
   volume={57},
   date={2018},
   number={1},
   pages={Paper No. 18, 42},
   issn={0944-2669},
   review={\MR{3740398}},
}

\bib{kim2009smooth}{article}{
   author={Kim, Sunghoon},
   author={Lee, Ki-Ahm},
   title={Smooth solution for the porous medium equation in a bounded
   domain},
   journal={J. Differential Equations},
   volume={247},
   date={2009},
   number={4},
   pages={1064--1095},
   issn={0022-0396},
   review={\MR{2531171}},
}

\bib{koch1999noneuclidean}{thesis}{
   author={Koch, Herbert},
   title={Non-Euclidean singular integrals and the porous medium equation},
   type={Habilitation},
   organization={Universit\"at Heidelberg},
   place={Heidelberg},
   date={1999},
}

\bib{ladyzenskaja1967linear}{book}{
   author={Lady\v zenskaja, O. A.},
   author={Solonnikov, V. A.},
   author={Ural\cprime ceva, N. N.},
   title={Linear and quasilinear equations of parabolic type},
   language={Russian},
   series={Translations of Mathematical Monographs},
   volume={Vol. 23},
   note={Translated from the Russian by S. Smith},
   publisher={American Mathematical Society, Providence, RI},
   date={1968},
   pages={xi+648},
   review={\MR{0241822}},
}

\bib{lee2025boundary}{article}{
   author={Lee, Ki-Ahm},
   author={Yun, Hyungsung},
   title={Boundary regularity for viscosity solutions of fully nonlinear
   degenerate/singular parabolic equations},
   journal={Calc. Var. Partial Differential Equations},
   volume={64},
   date={2025},
   number={1},
   pages={Paper No. 25, 32},
   issn={0944-2669},
   review={\MR{4836081}},
}

\bib{oleinik1958cauchy}{article}{
   author={Ole\u inik, O. A.},
   author={Kala\v sinkov, A. S.},
   author={\v C\v zou, Yu\u i-Lin\cprime},
   title={The Cauchy problem and boundary problems for equations of the type
   of non-stationary filtration},
   language={Russian},
   journal={Izv. Akad. Nauk SSSR Ser. Mat.},
   volume={22},
   date={1958},
   pages={667--704},
   issn={0373-2436},
   review={\MR{0099834}},
}

\bib{otto1996contraction}{article}{
   author={Otto, F.},
   title={$L^1$-contraction and uniqueness for quasilinear
   elliptic-parabolic equations},
   date={1996},
   ISSN={0022-0396,1090-2732},
   journal={J. Differential Equations},
   volume={131},
   number={1},
   pages={20\ndash 38},
   url={https://doi.org/10.1006/jdeq.1996.0155},
   review={\MR{1415045}},
}

\bib{pierre1987nonlinear}{article}{
   author={Pierre, M.},
   title={Nonlinear fast diffusion with measures as data},
   conference={
      title={Nonlinear parabolic equations: qualitative properties of
      solutions},
      address={Rome},
      date={1985},
   },
   book={
      series={Pitman Res. Notes Math. Ser.},
      volume={149},
      publisher={Longman Sci. Tech., Harlow},
   },
   isbn={0-582-99459-4},
   date={1987},
   pages={179--188},
   review={\MR{0901108}},
}


\bib{sabinina1962nonlinear}{article}{
   author={Sabinina, E. S.},
   title={On a class of non-linear degenerate parabolic equations},
   journal={Dokl. Akad. Nauk SSSR},
   volume={143:4},
   date={1962},
   pages={794-797},
}

\bib{sabinina1965quasilinear}{article}{
   author={Sabinina, E. S.},
   title={On a class of quasilinear parabolic equations, not solvable for
   the time derivative},
   journal={Sibirsk. Mat. Zh.},
   volume={6:5},
   date={1965},
   pages={1074-1100},
}

\bib{sire2022extinction}{article}{
   author={Sire, Yannick},
   author={Wei, Juncheng},
   author={Zheng, Youquan},
   title={Extinction behavior for the fast diffusion equations with critical
   exponent and Dirichlet boundary conditions},
   journal={J. Lond. Math. Soc. (2)},
   volume={106},
   date={2022},
   number={2},
   pages={855--898},
   issn={0024-6107},
   review={\MR{4477206}},
}

\bib{vazquez1992nonexistence}{article}{
   author={V\'azquez, Juan Luis},
   title={Nonexistence of solutions for nonlinear heat equations of
   fast-diffusion type},
   journal={J. Math. Pures Appl. (9)},
   volume={71},
   date={1992},
   number={6},
   pages={503--526},
   issn={0021-7824},
   review={\MR{1193606}},
}

\bib{vazquez2003darcy}{article}{
   author={V\'azquez, Juan L.},
   title={Darcy's law and the theory of shrinking solutions of fast
   diffusion equations},
   journal={SIAM J. Math. Anal.},
   volume={35},
   date={2003},
   number={4},
   pages={1005--1028},
   issn={0036-1410},
   review={\MR{2049031}},
}

\bib{vazquez2004dirichlet}{article}{
   author={Vazquez, Juan Luis},
   title={The Dirichlet problem for the porous medium equation in bounded
   domains. Asymptotic behavior},
   journal={Monatsh. Math.},
   volume={142},
   date={2004},
   number={1-2},
   pages={81--111},
   issn={0026-9255},
   review={\MR{2065023}},
}

\bib{vazquez2006smoothing}{book}{
   author={V\'azquez, Juan Luis},
   title={Smoothing and decay estimates for nonlinear diffusion equations},
   series={Oxford Lecture Series in Mathematics and its Applications},
   volume={33},
   note={Equations of porous medium type},
   publisher={Oxford University Press, Oxford},
   date={2006},
   pages={xiv+234},
   isbn={978-0-19-920297-3},
   isbn={0-19-920297-4},
   review={\MR{2282669}},
   doi={10.1093/acprof:oso/9780199202973.001.0001},
}

\bib{vazquez2007porous}{book}{
   author={V\'azquez, Juan Luis},
   title={The porous medium equation},
   series={Oxford Mathematical Monographs},
   note={Mathematical theory},
   publisher={The Clarendon Press, Oxford University Press, Oxford},
   date={2007},
   pages={xxii+624},
   isbn={978-0-19-856903-9},
   isbn={0-19-856903-3},
   review={\MR{2286292}},
}

\bib{yun2024regularity}{article}{
   author={Yun, Hyungsung},
   title={Regularity theory for fully nonlinear equations of porous
   medium-type},
   date={2024},
   journal={arXiv:2407.20022},
   url={https://arxiv.org/abs/2407.20022},
}











\end{biblist}
\end{bibdiv}


\bigskip
\bigskip

\noindent T. Jin

\noindent Department of Mathematics, The Hong Kong University of Science and Technology\\
Clear Water Bay, Kowloon, Hong Kong\\
Email: \textsf{tianlingjin@ust.hk}

\medskip

\noindent X. Tu

\noindent Department of Mathematics, The Hong Kong University of Science and Technology\\
Clear Water Bay, Kowloon, Hong Kong\\[1mm]
Email:  \textsf{maxstu@ust.hk}, \textsf{tuxushan28561m@gmail.com}

\medskip

\noindent J. Xiong

\noindent School of Mathematical Sciences, Laboratory of Mathematics and Complex Systems, MOE\\ Beijing Normal University, 
Beijing 100875, China\\
Email: \textsf{jx@bnu.edu.cn}

\medskip

\noindent Z. Zheng

\noindent Department of Mathematics, The Hong Kong University of Science and Technology\\
Clear Water Bay, Kowloon, Hong Kong\\[1mm]
Email:  \textsf{zzhengax@connect.ust.hk}

\end{document}